\documentclass[12pt]{article}
\usepackage{fullpage}
\usepackage{amsmath,amssymb,amsfonts,amsthm}
\usepackage[all,arc,knot,poly]{xy}
\usepackage{enumerate}
\usepackage{mathrsfs}
\usepackage{mathtools}
\usepackage{amsxtra}
\usepackage{needspace}
\usepackage{graphicx}
\usepackage{setspace}

\DeclareMathOperator{\Crit}{\operatorname{Crit}}
\DeclareMathOperator{\Zer}{\operatorname{Zer}}
\DeclareMathOperator{\Pol}{\operatorname{Pol}}
\DeclareMathOperator{\Quad}{\operatorname{Quad}}
\DeclareMathOperator{\Hom}{\operatorname{Hom}}
\DeclareMathOperator{\Ext}{\operatorname{Ext}}
\DeclareMathOperator{\Tw}{\operatorname{Tw}}
\DeclareMathOperator{\Stab}{\operatorname{Stab}}
\DeclareMathOperator{\Aut}{\operatorname{Aut}}
\DeclareMathOperator{\Sph}{\operatorname{Sph}}
\DeclareMathOperator{\sgn}{\operatorname{sgn}}

\theoremstyle{definition}
\newtheorem{theorem}{Theorem}[section]
\newtheorem{definition}[theorem]{Definition}
\newtheorem{lemma}[theorem]{Lemma}
\newtheorem{proposition}[theorem]{Proposition}

\title{Stability conditions and cluster varieties \protect\\from quivers of type~$A$}

\author{
  Dylan G.L. Allegretti
}

\date{}

\begin{document}

\maketitle

\begin{abstract}
We describe the relationship between two spaces associated to a quiver with potential. The first is a complex manifold parametrizing Bridgeland stability conditions on a triangulated category, and the second is a cluster variety with a natural Poisson structure. For quivers of type $A$, we construct a local biholomorphism from the space of stability conditions to the cluster variety. The existence of this map follows from results of Sibuya in the classical theory of ordinary differential equations.
\end{abstract}

\tableofcontents

\section{Introduction}

\subsection{A map from the space of stability conditions to the cluster variety}

This paper is concerned with the relationship between two spaces associated to a quiver with potential. The first is a complex manifold parametrizing Bridgeland stability conditions on a triangulated category, and the second is a geometric object known as the cluster Poisson variety. The structure of each space is controlled by the combinatorics of the exchange graph of the quiver, but the combinatorics is used quite differently in the two cases. Indeed, the space of stability conditions has a wall-and-chamber decomposition, whereas the cluster variety is defined as a union of algebraic tori glued by birational maps.

In a forthcoming paper~\cite{AllegrettiBridgeland}, we study the relationship between these spaces for a large class of quivers with potentials associated to triangulated surfaces. In the present paper, we focus on the special case where the underlying graph of the quiver is a Dynkin diagram of type~$A_n$. We construct a natural map 
\[
F:\Sigma(A_n)\rightarrow\mathcal{X}(A_n)
\]
from the space of stability conditions to the cluster variety, and we prove that it is locally a biholomorphism onto its image.

To construct our map, we make use of the explicit descriptions of the spaces $\Sigma(A_n)$ and $\mathcal{X}(A_n)$ due to Bridgeland and Smith~\cite{BridgelandSmith} and Fock and Goncharov~\cite{FG1}, respectively. These descriptions allow us to identify $\Sigma(A_n)$ with a space of meromorphic quadratic differentials and to identify $\mathcal{X}(A_n)$ with a certain moduli space of generic configurations of points in~$\mathbb{CP}^1$. Once we have reinterpreted the two spaces in this way, the map~$F$ turns out to be equivalent to a construction of Sibuya~\cite{Sibuya} in the classical theory of ordinary differential equations.

Our use of analytic methods to construct the map $F$ is inspired by the work of Gaiotto, Moore, and Neitzke~\cite{GMN2} who described a similar relationship between the space of Higgs bundles on a Riemann surface and the cluster variety associated to a triangulation of the surface. Their construction was designed to give solutions to the Riemann-Hilbert problems formulated in~\cite{GMN1}. Likewise, the map $F$ is closely related to the Riemann-Hilbert problems studied by Bridgeland~\cite{Bridgeland16} in the context of Donaldson-Thomas theory. We will revisit this aspect of our construction in the sequel paper~\cite{AllegrettiBridgeland}.

The map $F$ is also related to the map recently introduced by Goncharov in~\cite{Goncharov}, which takes stability conditions to real tropical points of the cluster variety. Indeed, the map constructed in~\cite{Goncharov} appears to be a kind of tropical limit of the one studied in the present paper.

In the remainder of this introduction, we will describe our main results in more detail.

\subsection{Asymptotics of solutions to differential equations}

The space $\Sigma(A_n)$ parametrizes Bridgeland stability conditions on a triangulated category $\mathcal{D}$ associated to a quiver of type~$A_n$ by the Ginzburg algebra construction. More precisely, the space we consider is a quotient 
\[
\Sigma(A_n)=\Stab^\circ(\mathcal{D})/\Sph(\mathcal{D})
\]
where $\Stab^\circ(\mathcal{D})$ is a connected component of the usual manifold of stability conditions~\cite{Bridgeland07} and $\Sph(\mathcal{D})$ is a subgroup of the group of exact autoequivalences of~$\mathcal{D}$ generated by functors called \emph{spherical twists}~\cite{SeidelThomas}. The details of this definition will be explained in Section~\ref{sec:StabilityConditionsAndTheClusterVariety}. As we have indicated, points of~$\Sigma(A_n)$ correspond to meromorphic quadratic differentials, that is, meromorphic sections of $\omega_{\mathbb{CP}^1}^{\otimes2}$, the square of the holomorphic cotangent bundle of~$\mathbb{CP}^1$. Let $z$ be the local coordinate corresponding to the complex plane. The space of stability conditions corresponds to the particular class of quadratic differentials which can be represented by an expression 
\[
\phi(z)=P(z)dz^{\otimes2}
\]
where $P(z)$ is a polynomial of the form 
\[
P(z)=z^{n+1}+a_{n-1}z^{n-1}+\dots,+a_1z+a_0
\]
with simple zeros and $a_0,\dots,a_{n-1}\in\mathbb{C}$.

To apply the results of Sibuya, we form the second order differential equation 
\begin{align}
\label{eqn:introschrodinger}
y''(z)-P(z)y(z)=0.
\end{align}
This equation has an irregular singularity at the point $\infty$. Consequently, the asymptotics of its solutions depend on the direction in which $z$ tends to infinity in the complex plane. To understand this behavior more precisely, we introduce, for each $k\in\mathbb{Z}/(n+3)\mathbb{Z}$, the $k$th \emph{Stokes sector}:
\[
\mathscr{S}_k=\left\{z\in\mathbb{C}:\left|\arg z-\frac{2\pi}{n+3}k\right|<\frac{\pi}{n+3}\right\}.
\]
For example, the diagram below illustrates the Stokes sectors for $n=2$:
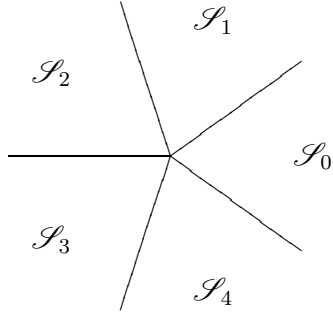
\[
\xy /l1.5pc/:
{\xypolygon5"A"{~:{(2,2.75):}~>{}}};
{(1,0)\PATH~={**@{-}}'"A1"};
{(1,0)\PATH~={**@{-}}'"A2"};
{(1,0)\PATH~={**@{-}}'"A3"};
{(1,0)\PATH~={**@{-}}'"A4"};
{(1,0)\PATH~={**@{-}}'"A5"};
(-2,0)*{\mathscr{S}_0};
(0.1,2.85)*{\mathscr{S}_4};
(0.1,-2.85)*{\mathscr{S}_1};
(3.5,-1.75)*{\mathscr{S}_2};
(3.5,1.75)*{\mathscr{S}_3};
\endxy
\]
In~\cite{Sibuya}, Sibuya considered, for any pair of linearly independent solutions $y_1(z)$ and $y_2(z)$ of~\eqref{eqn:introschrodinger}, the values 
\[
z_k\coloneqq\lim_{\substack{z\rightarrow\infty \\ z\in\mathscr{S}_k}}\frac{y_1(z)}{y_2(z)}
\]
where the limit is taken along any ray in the $k$th Stokes sector $\mathscr{S}_k$.

One can show that the limits $z_k$ for $k\in\mathbb{Z}/(n+3)\mathbb{Z}$ exist in~$\mathbb{CP}^1$ independently of the direction in $\mathscr{S}_k$ in which the limit is taken. Moreover, if we start with a different pair of linearly independent solutions, this construction yields a different set of values $z_k'\in\mathbb{CP}^1$, which are related to the~$z_k$ by the action of a common element of $PGL_2(\mathbb{C})$. In other words, the differential equation determines a well defined point in the quotient 
\begin{align}
\label{eqn:quotientspace}
(\mathbb{CP}^1)^{n+3}/PGL_2(\mathbb{C}).
\end{align}
The latter space is familiar from the work of Fock and Goncharov on cluster varieties. Indeed, it is a special case of their moduli space of framed $PGL_2(\mathbb{C})$-local systems on a surface~\cite{FG1}. In particular, this space is equipped with an atlas of rational coordinates given by cross ratios of points in~$\mathbb{CP}^1$. The idea that the space of monodromy data is coordinatized by cross ratios was also observed in the differential equations literature by Masoero~\cite{Masoero}.

It turns out that the values $z_k$ determined by the above construction cannot be arbitrary and must satisfy a certain genericity condition formulated by Sibuya. Remarkably, the locus of points in the above quotient satisfying this genericity condition is precisely the cluster Poisson variety. Thus we see that a choice of stability condition determines the differential equation~\eqref{eqn:introschrodinger}, and there is a corresponding point in the cluster variety determined by the asymptotic behavior of its solutions. This defines the map $F$. It follows from work of Bakken~\cite{Bakken} (see also~\cite{Sibuya}) that this map is a local biholomorphism. It intertwines natural actions of the group $\mathbb{Z}/(n+3)\mathbb{Z}$ on the spaces $\Sigma(A_n)$ and $\mathcal{X}(A_n)$.

\subsection{Wall-and-chamber decomposition and gluing of tori}

An important feature of the space $\Sigma(A_n)$ is that this space is composed of chambers glued together along their codimension~1 boundaries. This wall-and-chamber decomposition can be seen at the level of quadratic differentials by considering their induced foliations, as we now explain.

Let $\phi$ be a quadratic differential. Away from the critical points of $\phi$, there is a distinguished local coordinate $w$, unique up to transformations of the form $w\mapsto\pm w+\text{constant}$, such that 
\[
\phi(w)=dw^{\otimes2}.
\]
The \emph{horizontal foliation} is given by the lines $\Im(w)=\text{constant}$, and each line satisfying this condition is called a \emph{horizontal trajectory}. The particular quadratic differentials that we consider determine three types of horizontal trajectories:
\begin{enumerate}
\item A \emph{saddle trajectory} is a trajectory which connects distinct zeros of $\phi$.
\item A \emph{separating trajectory} is a trajectory which connects a zero and a pole of $\phi$.
\item A \emph{generic trajectory} is a trajectory which connects poles of $\phi$.
\end{enumerate}
For any quadratic differential, the number of saddle trajectories and separating trajectories is finite. There is a stratification of the space of quadratic differentials by the number of saddle trajectories. If $\phi$ lies in the top-dimensional stratum consisting of quadratic differentials with no saddle trajectories, then the separating trajectories define a decomposition of $\mathbb{CP}^1$ into finitely many cells, examples of which are illustrated below.
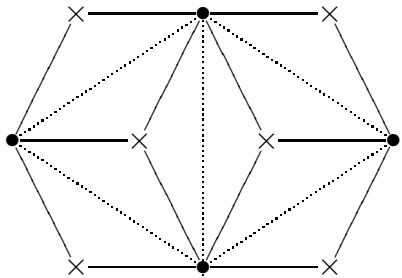
\[
\xy /l2pc/:
(2,-2)*{\times}="11";
(0,-2)*{\bullet}="21";
(-2,-2)*{\times}="31";
(3,0)*{\bullet}="12";
(1,0)*{\times}="22";
(-1,0)*{\times}="32";
(-3,0)*{\bullet}="42";
(2,2)*{\times}="13";
(0,2)*{\bullet}="23";
(-2,2)*{\times}="33";
{"11"\PATH~={**@{-}}'"21"},
{"21"\PATH~={**@{-}}'"31"},
{"12"\PATH~={**@{-}}'"22"},
{"32"\PATH~={**@{-}}'"42"},
{"13"\PATH~={**@{-}}'"23"},
{"23"\PATH~={**@{-}}'"33"},
{"11"\PATH~={**@{-}}'"12"},
{"21"\PATH~={**@{-}}'"22"},
{"21"\PATH~={**@{-}}'"32"},
{"31"\PATH~={**@{-}}'"42"},
{"12"\PATH~={**@{-}}'"13"},
{"22"\PATH~={**@{-}}'"23"},
{"32"\PATH~={**@{-}}'"23"},
{"42"\PATH~={**@{-}}'"33"},
{"12"\PATH~={**@{.}}'"21"},
{"21"\PATH~={**@{.}}'"42"},
{"12"\PATH~={**@{.}}'"23"},
{"23"\PATH~={**@{.}}'"42"},
{"21"\PATH~={**@{.}}'"23"},
\endxy
\]
Here we indicate zeros of the quadratic differential by $\times$ and poles by $\bullet$. The solid lines indicate separating trajectories. The quadratic differentials that we consider have a single pole of order $n+5$ at~$\infty$. It is known that the trajectories of a quadratic differential approach such a pole along $n+3$ distinguished tangent directions. Thus, if we take an oriented real blowup of~$\mathbb{CP}^1$ at the point~$\infty$, we get a disk, and the distinguished tangent directions determine $n+3$ marked points on the boundary of this disk. Choosing a single generic trajectory in each cell as illustrated by the dotted lines in the picture above, we obtain a triangulation of the disk, well defined up to isotopy, with vertices at the marked points. This triangulation is known as the \emph{WKB triangulation} of the quadratic differential. Each chamber in the space $\Sigma(A_n)$ can be understood as the set of all quadratic differentials inducing a given WKB triangulation.

Like the space of stability conditions, the structure of the cluster Poisson variety is controlled by the combinatorics of triangulations. To better understand this structure, consider a disk with $n+3$ marked points on its boundary. We can think of any point of the space~\eqref{eqn:quotientspace} as an assignment of a point of $\mathbb{CP}^1$ to each marked point, modulo the action of $PGL_2(\mathbb{C})$. Let $T$ be a triangulation of this disk with vertices at the marked points. We say that a point $\psi$ of the space~\eqref{eqn:quotientspace} is \emph{generic} with respect to an edge $i$ of this triangulation if the points of~$\mathbb{CP}^1$ assigned to the endpoints of $i$ by $\psi$ are distinct. We say that $\psi$ is \emph{generic} with respect to $T$ if it is generic with respect to every edge of $T$. We will see that the set of points $\mathcal{X}_T$ that are generic with respect to a triangulation $T$ is an algebraic torus:
\[
\mathcal{X}_T\cong(\mathbb{C}^*)^n.
\]
The cluster Poisson variety is the union of these tori for all triangulations $T$.

Consider, for each $\epsilon>0$, the region 
\[
\mathbb{H}(\epsilon)=\{\hbar\in\mathbb{C}:|\hbar|<\epsilon \text{ and } \Re(\hbar)>0\}.
\]
There is a natural $\mathbb{C}$-action on the space of stability conditions. (Specifically, a complex number $z\in\mathbb{C}$ acts on a stability condition with central charge $Z$ to give a stability condition with central charge $e^{-i\pi z}Z$; see Section~\ref{sec:StabilityConditionsAndTheClusterVariety}.) For any $\hbar\in\mathbb{H}(\epsilon)$, let $z$ be the unique complex number with $-\pi/2<\arg(z)<\pi/2$ such that $\hbar=e^{i\pi z}$ and define 
\[
F_\hbar:\Sigma(A_n)\rightarrow\mathcal{X}(A_n)
\]
by $F_\hbar(\sigma)=F(z\cdot\sigma)$. The following result describes the relationship between the wall-and-chamber decomposition of the space of stability conditions and the tori used to construct the cluster variety.

\begin{theorem}
\label{thm:introchambertotorus}
Suppose $\sigma$ lies in a chamber of $\Sigma(A_n)$. Then there exists $\epsilon>0$ such that for all points $\hbar\in\mathbb{H}(\epsilon)$, the point $F_\hbar(\sigma)$ is generic with respect to the WKB triangulation associated with $\sigma$.
\end{theorem}

The proof of this result is essentially a consequence of the WKB approximation in the theory of differential equations.

\subsection{Organization}

We begin in Section~\ref{sec:QuadraticDifferentials} by reviewing the relevant facts about quadratic differentials and the combinatorics of triangulations. In Section~\ref{sec:ConfigurationsOfPointsInCP1}, we discuss configurations of points in~$\mathbb{CP}^1$ and their parametrization by cross ratios. In Section~\ref{sec:TheMainConstruction}, we define the map $F$ using results of Sibuya and note that it is a local biholomorphism. We also review the theory of exact WKB analysis and use it to prove a version of Theorem~\ref{thm:introchambertotorus}. Finally, in Section~\ref{sec:StabilityConditionsAndTheClusterVariety}, we define stability conditions and the cluster variety in the abstract setting of triangulated categories, and we interpret our main results in this setting.

\section{Quadratic differentials}
\label{sec:QuadraticDifferentials}

\subsection{The trajectory structure of quadratic differentials}

In this section, we review the ideas we need from the theory of quadratic differentials.

\begin{definition}
A meromorphic \emph{quadratic differential} is a meromorphic section of $\omega_{\mathbb{CP}^1}^{\otimes2}$ where $\omega_{\mathbb{CP}^1}$ is the holomorphic cotangent bundle of~$\mathbb{CP}^1$.
\end{definition}

For any local coordinate $z$ on the Riemann sphere, we can represent a quadratic differential $\phi$ by an expression 
\[
\phi(z)=\varphi(z)dz^{\otimes2}
\]
where $\varphi$ is a meromorphic function. If $\phi(w)=\tilde{\varphi}(w)dw^{\otimes2}$ is the expression for $\phi$ in a different local coordinate $w$, then we have $\varphi(z)=\tilde{\varphi}(w)(dw/dz)^2$.

\begin{definition}
If $\phi$ is a quadratic differential, then a zero or pole of $\phi$ is called a \emph{critical point}. We will denote the set of all critical points of $\phi$ by $\Crit(\phi)$. We will denote the subsets of zeros and poles of $\phi$ by $\Zer(\phi)$ and $\Pol(\phi)$, respectively. A zero or simple pole of~$\phi$ is called a \emph{finite critical point}, and any other critical point is called an \emph{infinite critical point}.
\end{definition}

In a neighborhood of any point which is not a critical point of $\phi$, there is a distinguished local coordinate $w$, unique up to transformations of the form $w\mapsto\pm w+\text{constant}$, such that 
\[
\phi(w)=dw^{\otimes2}.
\]
Indeed, if we have $\phi(z)=\varphi(z)dz^{\otimes2}$ for some local coordinate $z$ away from the critical points, then we can define $w$ by 
\[
w(z)=\int^z\sqrt{\varphi(z)}dz
\]
for some choice of square root of $\varphi(z)$. Using this local coordinate, we can define the foliation induced by a quadratic differential.

\begin{definition}
If $\phi$ is a quadratic differential, then a \emph{horizontal trajectory} of $\phi$ is a curve in $\mathbb{CP}^1\setminus\Crit(\phi)$ given by $\Im(w)=\text{constant}$ where $w$ is the distinguished local coordinate. The \emph{horizontal foliation} is the foliation of $\mathbb{CP}^1\setminus\Crit(\phi)$ by horizontal trajectories.
\end{definition}

Let us consider the local behavior of the horizontal trajectories near a critical point of a quadratic differential. First, suppose $p\in\Zer(\phi)$ is a zero of order~$k\geq1$. It is known (see~\cite{Strebel}, Section~6) that there exists a local coordinate $t$ such that $t(p)=0$ and 
\[
\phi(t)=\left(\frac{k+2}{2}\right)^2t^kdt^{\otimes2}.
\]
It follows that, in a neighborhood of~$p$, the distinguished local coordinate is $w=t^{\frac{1}{2}(k+2)}$. The horizontal trajectories determined by this local coordinate form a $(k+2)$-pronged singularity as illustrated below for $k=1,2$.
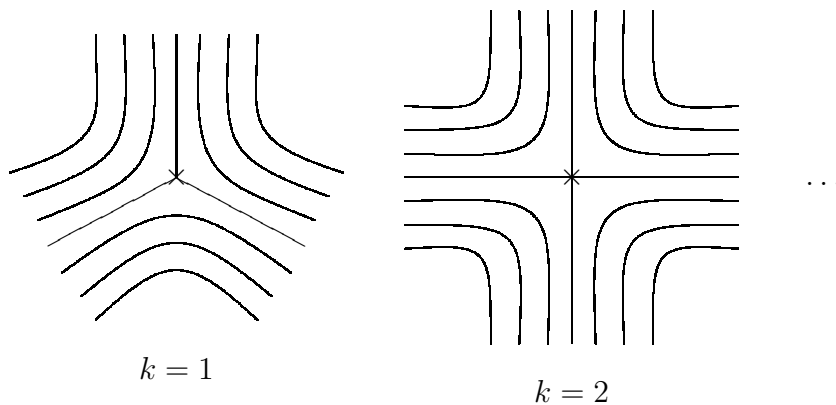
\[
\xy /l3pc/:
(1,0)*{}="O"; 
(-0.35,0.72)*{}="U"; 
(-0.75,-0.05)*{}="X1";
(-0.6,0.2)*{}="X2";
(-0.45,0.45)*{}="X3"; 
(-0.2,1)*{}="X4"; 
(0,1.25)*{}="X5"; 
(0.15,1.5)*{}="X6"; 
(1.85,1.5)*{}="Y1";
(2,1.25)*{}="Y2";
(2.2,1)*{}="Y3";
(2.45,0.45)*{}="Y4";
(2.6,0.2)*{}="Y5";
(2.75,-0.05)*{}="Y6";
(1.85,-1.5)*{}="Z1";
(1.55,-1.5)*{}="Z2";
(1.25,-1.5)*{}="Z3";
(0.75,-1.5)*{}="Z4";
(0.45,-1.5)*{}="Z5";
(0.15,-1.5)*{}="Z6";
(2.35,0.72)*{}="V"; 
(1,-1.5)*{}="W"; 
"O";"U" **\dir{-}; 
"O";"V" **\dir{-}; 
"O";"W" **\dir{-}; 
"X4";"Y3" **\crv{(0.9,0.2) & (1.1,0.2)};
"X5";"Y2" **\crv{(0.9,0.5) & (1.1,0.5)};
"X6";"Y1" **\crv{(0.9,0.8) & (1.1,0.8)};
"Y4";"Z3" **\crv{(1.35,0) & (1.15,0)};
"Y5";"Z2" **\crv{(1.5,-0.2) & (1.5,-0.3)};
"Y6";"Z1" **\crv{(1.65,-0.4) & (1.85,-0.6)};
"Z4";"X3" **\crv{(0.85,0) & (0.65,0)};
"Z5";"X2" **\crv{(0.5,-0.3) & (0.5,-0.2)};
"Z6";"X1" **\crv{(0.15,-0.6) & (0.35,-0.4)};
(1,0)*{\times};
(1,2)*{k=1};
\endxy
\qquad
\xy /l3pc/:
(1,0)*{}="O"; 
(1,-1.75)*{}="T"; 
(-0.75,0)*{}="U"; 
(1,1.75)*{}="V"; 
(2.75,0)*{}="W"; 
(-0.75,-0.75)*{}="U1";
(-0.75,-0.5)*{}="U2";
(-0.75,-0.25)*{}="U3"; 
(-0.75,0.25)*{}="U4";
(-0.75,0.5)*{}="U5";
(-0.75,0.75)*{}="U6";
(0.15,1.75)*{}="V1";
(0.45,1.75)*{}="V2";
(0.75,1.75)*{}="V3";
(1.25,1.75)*{}="V4"; 
(1.55,1.75)*{}="V5"; 
(1.85,1.75)*{}="V6"; 
(2.75,0.75)*{}="W1";
(2.75,0.5)*{}="W2";
(2.75,0.25)*{}="W3";
(2.75,-0.25)*{}="W4";
(2.75,-0.5)*{}="W5";
(2.75,-0.75)*{}="W6";
(1.85,-1.75)*{}="T1";
(1.55,-1.75)*{}="T2";
(1.25,-1.75)*{}="T3";
(0.75,-1.75)*{}="T4";
(0.45,-1.75)*{}="T5";
(0.15,-1.75)*{}="T6";
"O";"T" **\dir{-}; 
"O";"U" **\dir{-}; 
"O";"V" **\dir{-}; 
"O";"W" **\dir{-}; 
"U4";"V3" **\crv{(0.8,0.2) & (0.8,0.2)};
"U5";"V2" **\crv{(0.5,0.5) & (0.5,0.5)};
"U6";"V1" **\crv{(0.15,0.7) & (0.15,0.7)};
"V4";"W3" **\crv{(1.2,0.2) & (1.2,0.2)};
"V5";"W2" **\crv{(1.5,0.5) & (1.5,0.5)};
"V6";"W1" **\crv{(1.85,0.7) & (1.85,0.7)};
"W4";"T3" **\crv{(1.2,-0.2) & (1.2,-0.2)};
"W5";"T2" **\crv{(1.5,-0.5) & (1.5,-0.5)};
"W6";"T1" **\crv{(1.85,-0.7) & (1.85,-0.7)};
"T4";"U3" **\crv{(0.8,-0.2) & (0.8,-0.2)};
"T5";"U2" **\crv{(0.5,-0.5) & (0.5,-0.5)};
"T6";"U1" **\crv{(0.15,-0.7) & (0.15,-0.7)};
(1,0)*{\times};
(1,2.25)*{k=2};
\endxy
\qquad
\dots
\]

On the other hand, suppose that $p\in\Pol(\phi)$ is a pole of order $m>2$. In this case, a similar argument (see~\cite{Strebel}, Section~6) shows that there is a neighborhood $U$ of $p$ and a collection of $m-2$ distinguished tangent directions $v_i$ at $p$ such that any horizontal trajectory that enters $U$ eventually tends to $p$ and is asymptotic to one of the $v_i$. We illustrate this below for $m=5,6$.
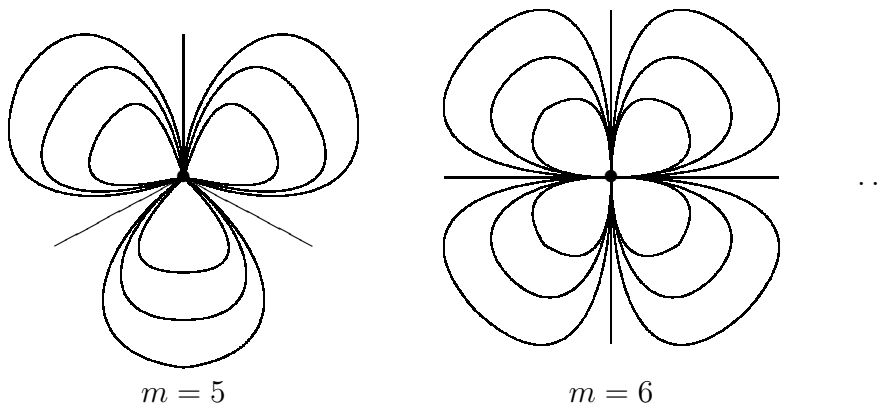
\[
\xy /l3pc/:
{\xypolygon3"T"{~:{(2,0):}~>{}}},
{\xypolygon3"S"{~:{(1.5,0):}~>{}}},
{\xypolygon3"R"{~:{(1,0):}~>{}}},
(1,0)*{}="O"; 
(-0.35,0.72)*{}="U"; 
(2.35,0.72)*{}="V"; 
(1,-1.5)*{}="W"; 
"O";"U" **\dir{-}; 
"O";"V" **\dir{-}; 
"O";"W" **\dir{-}; 
"O";"T1" **\crv{(2,0.75) & (2.25,1.85)};
"O";"T1" **\crv{(0,0.75) & (-0.25,1.85)};
"O";"T2" **\crv{(0,0.4) & (-1.2,0.25)};
"O";"T2" **\crv{(1,-1) & (0,-2.2)};
"O";"T3" **\crv{(1,-1) & (2,-2.2)};
"O";"T3" **\crv{(2,0.4) & (3.2,0.25)};
"O";"S1" **\crv{(1.75,0.56) & (2,1.5)};
"O";"S1" **\crv{(0.25,0.56) & (0,1.5)};
"O";"S2" **\crv{(0,0.3) & (-0.9,0.19)};
"O";"S2" **\crv{(0.9,-0.8) & (0.3,-1.7)};
"O";"S3" **\crv{(2,0.3) & (2.9,0.19)};
"O";"S3" **\crv{(1.1,-0.8) & (1.7,-1.7)};
"O";"R1" **\crv{(1.5,0.5) & (1.75,1)};
"O";"R1" **\crv{(0.5,0.5) & (0.25,1)};
"O";"R2" **\crv{(0.5,0.1) & (-0.3,0.25)};
"O";"R2" **\crv{(0.75,-0.8) & (0.5,-1)};
"O";"R3" **\crv{(1.5,0.1) & (2.3,0.25)};
"O";"R3" **\crv{(1.25,-0.8) & (1.5,-1)};
(1,0)*{\bullet};
(1,2.25)*{m=5};
\endxy
\qquad
\xy /l3pc/:
{\xypolygon4"A"{~:{(2,0):}~>{}}},
{\xypolygon4"B"{~:{(1.5,0):}~>{}}},
{\xypolygon4"C"{~:{(1,0):}~>{}}},
(1,0)*{}="O"; 
(1,-1.75)*{}="T"; 
(-0.75,0)*{}="U"; 
(1,1.75)*{}="V"; 
(2.75,0)*{}="W"; 
"O";"T" **\dir{-}; 
"O";"U" **\dir{-}; 
"O";"V" **\dir{-}; 
"O";"W" **\dir{-}; 
"O";"A1" **\crv{(1,1.5) & (1.5,2.25)};
"O";"A1" **\crv{(2.5,0) & (3.25,0.5)};
"O";"A2" **\crv{(1,1.5) & (0.5,2.25)};
"O";"A2" **\crv{(-0.5,0) & (-1.25,0.5)};
"O";"A3" **\crv{(1,-1.5) & (0.5,-2.25)};
"O";"A3" **\crv{(-0.5,0) & (-1.25,-0.5)};
"O";"A4" **\crv{(1,-1.5) & (1.5,-2.25)};
"O";"A4" **\crv{(2.5,0) & (3.25,-0.5)};
"O";"B1" **\crv{(1,1) & (1.5,1.6)};
"O";"B1" **\crv{(2,0) & (2.6,0.5)};
"O";"B2" **\crv{(1,1) & (0.5,1.6)};
"O";"B2" **\crv{(0,0) & (-0.6,0.5)};
"O";"B3" **\crv{(1,-1) & (0.5,-1.6)};
"O";"B3" **\crv{(0,0) & (-0.6,-0.5)};
"O";"B4" **\crv{(1,-1) & (1.5,-1.6)};
"O";"B4" **\crv{(2,0) & (2.6,-0.5)};
"O";"C1" **\crv{(1,0.75) & (1.25,1)};
"O";"C1" **\crv{(1.75,0) & (2,0.25)};
"O";"C2" **\crv{(1,0.75) & (0.75,1)};
"O";"C2" **\crv{(0.25,0) & (0,0.25)};
"O";"C3" **\crv{(1,-0.75) & (0.75,-1)};
"O";"C3" **\crv{(0.25,0) & (0,-0.25)};
"O";"C4" **\crv{(1,-0.75) & (1.25,-1)};
"O";"C4" **\crv{(1.75,0) & (2,-0.25)};
(1,0)*{\bullet};
(1,2.25)*{m=6};
\endxy
\qquad
\dots
\]

Next we will examine the global structure of the horizontal foliation. Three types of horizontal trajectories will play a role in our analysis.

\begin{definition}
Let $\phi$ be a quadratic differential.
\begin{enumerate}
\item A \emph{saddle trajectory} is a horizontal trajectory which connects finite critical points of~$\phi$.
\item A \emph{separating trajectory} is a horizontal trajectory which connects a finite and an infinite critical point of~$\phi$.
\item A \emph{generic trajectory} is a horizontal trajectory which connects infinite critical points of~$\phi$.
\end{enumerate}
\end{definition}

As explained in the introduction, any quadratic differential that corresponds to a stability condition in our setting has simple zeros and a single pole of order $n+5$ in~$\mathbb{CP}^1$. This leads to a simple classification of the possible horizontal trajectories.

\begin{proposition}
\label{prop:noclosedrecurrent}
Let $\phi$ be a quadratic differential with a single pole of order $n+5$ in~$\mathbb{CP}^1$. Then every horizontal trajectory of~$\phi$ is either a saddle trajectory, a separating trajectory, or a generic trajectory.
\end{proposition}

\begin{proof}
This follows from Theorem~14.2.2 of~\cite{Strebel}, which implies that there are no closed trajectories in the horizontal foliation, and Theorem~15.2 of~\cite{Strebel}, which implies that there are no recurrent trajectories.
\end{proof}

We will be interested in the following types of regions determined by the horizontal foliation.

\Needspace*{2\baselineskip}
\begin{definition} \mbox{}
\begin{enumerate}
\item A \emph{half plane} is a connected component of the complement of the separating trajectories in $\mathbb{CP}^1$ which is mapped by the distinguished local coordinate to 
\[
\{w\in\mathbb{C}:\Im(w)>0\}.
\]
The trajectories in a half plane are generic, connecting a fixed pole of order $>2$ to itself. The boundary is composed of saddle trajectories and separating trajectories.

\item A \emph{horizontal strip} is a connected component of the complement of the separating trajectories in $\mathbb{CP}^1$ which is mapped by the distinguished local coordinate to
\[
\{w\in\mathbb{C}:a<\Im(w)<b\}.
\]
The trajectories in a horizontal strip are generic, connecting two (not necessarily distinct) poles. The boundary is composed of saddle trajectories and separating trajectories.
\end{enumerate}
\end{definition}

We have seen that there are only finitely many horizontal trajectories incident to any zero of a quadratic differential. It follows that there are only finitely many saddle trajectories and separating trajectories. A quadratic differential will be called \emph{saddle-free} if it has no saddle trajectories.

\begin{proposition}
\label{prop:horizontalstripdecomposition}
Suppose $\phi$ is a quadratic differential with simple zeros and a single pole of order $n+5$. If $\phi$ is saddle-free, then after removing the finitely many separating trajectories, we obtain an open surface which is a union of horizontal strips and half planes.
\end{proposition}

\begin{proof}
This follows from Section~11.4 of~\cite{Strebel} and Proposition~\ref{prop:noclosedrecurrent}.
\end{proof}

\subsection{The moduli space of framed differentials}
\label{sec:TheModuliSpaceOfFramedDifferentials}

In this subsection, we explain how to associate a moduli space of quadratic differentials to a disk with marked points. Throughout our discussion, we will denote by $\mathbb{D}$ a closed disk and by $\mathbb{M}$ a set of $n+3$ distinct marked points contained in the boundary of~$\mathbb{D}$.

Suppose $\phi$ is a quadratic differential with simple zeros and a single pole $p$ of order $n+5$. We have seen that the horizontal foliation determines $n+3$ distinguished tangent directions at~$p$. Therefore, if we take an oriented real blowup at~$p$, we get a disk~$\widetilde{\mathbb{D}}$, and the distinguished tangent directions give rise to a set $\widetilde{\mathbb{M}}$ of $n+3$ marked points on the boundary of~$\widetilde{\mathbb{D}}$. A framing is an extra datum that allows us to compare this with the standard pair $(\mathbb{D},\mathbb{M})$ chosen above.

\begin{definition}
Let $\phi$ be a quadratic differential with simple zeros and a single a single pole of order $n+5$. A \emph{framing} for $\phi$ is defined to be an orientation preserving diffeomorphism $f:\widetilde{\mathbb{D}}\rightarrow\mathbb{D}$ inducing a bijection $\widetilde{\mathbb{M}}\cong\mathbb{M}$ and considered up to diffeomorphisms of~$\widetilde{\mathbb{D}}$ isotopic to the identity. Two framed differentials $(\phi_1,f_1)$ and $(\phi_2,f_2)$ are said to be \emph{equivalent} if there is a biholomorphism $g:\mathbb{CP}^1\rightarrow\mathbb{CP}^1$ satisfying $g^*\phi_2=\phi_1$ which commutes with the framing in the obvious way. We denote by $\Quad_{fr}(\mathbb{D},\mathbb{M})$ the space of equivalence classes of framed quadratic differentials.
\end{definition}

Let $(\phi,f)$ be a point in the space $\Quad_{fr}(\mathbb{D},\mathbb{M})$. By definition, $\phi$ is a quadratic differential with simple zeros and a single pole of order $n+5$. Since this is well defined only up to equivalence, we can assume that the pole is $\infty\in\mathbb{CP}^1$. Then we can write 
\begin{align}
\label{eqn:polynomialdifferential}
\phi(z)=P(z)dz^{\otimes2}
\end{align}
where $P(z)$ is a polynomial of degree $n+1$ having simple roots. After rescaling $z$ by an appropriate factor, we can assume that $P(z)$ is monic.

\begin{proposition}
\label{prop:distinguishedtangents}
If $\phi$ is the quadratic differential given by~\eqref{eqn:polynomialdifferential} where $P(z)$ is a monic polynomial of degree $n+1$, then the distinguished tangent vectors determined by~$\phi$ are the tangent vectors to the rays $\mathbb{R}_{>0}\exp(2\pi i k/(n+3))$ at $\infty$ for $k=1,\dots,n+3$.
\end{proposition}

\begin{proof}
Consider a tangent vector to $\infty$ which is not of the type appearing in the statement of the proposition. Such a tangent vector is tangent to a ray $\mathbb{R}_{>0}z_0$ where $|z_0|=1$ and $z_0$ is not an $(n+3)$rd root of unity. Let us study how the imaginary part of the distinguished local coordinate varies along this ray. Away from $\infty$, the differential $\phi$ can be described by the local expression~\eqref{eqn:polynomialdifferential}. The leading term of $P(z)$ is $z^{n+1}$, so this local coordinate can be written 
\[
w(z)=\frac{2}{n+3}z^{\frac{n+3}{2}}+\text{lower terms}.
\]
If we set $z=\lambda z_0$ for $\lambda\in\mathbb{R}_{>0}$, then 
\[
\left|\Im\left(\frac{2}{n+3}z^{\frac{n+3}{2}}\right)\right| = \lambda^{\frac{n+3}{2}}\left|\Im\left(\frac{2}{n+3}z_0^{\frac{n+3}{2}}\right)\right| \rightarrow\infty
\]
as $\lambda\rightarrow\infty$. Moreover, this expression tends to infinity faster than the imaginary part of any lower term of $w(z)$. It follows that the expression $|\Im(w(z))|$ increases without bound as $\lambda\rightarrow\infty$. In particular, the ray $\mathbb{R}_{>0}z_0$ cannot be asymptotic to any horizontal trajectory. This completes the proof.
\end{proof}

Fix a map $f_0$ from the set of $(n+3)$rd roots of unity to~$\mathbb{M}$ which preserves the cyclic ordering of these sets. By Proposition~\ref{prop:distinguishedtangents}, this map $f_0$ determines a framing of~$\phi$. After multiplying $z$ by an appropriate $(n+3)$rd root of unity, we can assume that the framing $f$ coincides with~$f_0$. Finally, by applying a translation to the coordinate~$z$, we can assume that the roots of $P(z)$ sum to zero. That is, 
\[
P(z)=\prod_{i=1}^{n+1}(z-\alpha_i)
\]
where 
\[
\sum_{i=1}^{n+1}\alpha_i=0.
\]
Expanding this product, we can write 
\[
P(z)=z^{n+1}+a_nz^n+a_{n-1}z^{n-1}+\dots+a_1z+a_0
\]
where $a_i\in\mathbb{C}$ for $i=0,\dots,n$, and we have $a_n=-\sum_{i=1}^{n+1}\alpha_i=0$. Thus the space of monic polynomials $P(z)$ of degree $n+1$ such that the sum of the roots of $P(z)$ is zero is identified with $\mathbb{C}^n=\{(a_0,\dots,a_{n-1}):a_i\in\mathbb{C}\text{ for all $i$}\}$. Since our polynomials have simple roots, we consider the Zariski open subset $\mathbb{C}^n\setminus\Delta$ where 
\[
\Delta=\left\{\prod_{i<j}(\alpha_i-\alpha_j)^2=0\right\}
\]
is the discriminant locus. Summarizing, we have the following result.

\begin{proposition}
\label{prop:identificationquad}
The moduli space $\Quad_{fr}(\mathbb{D},\mathbb{M})$ is isomorphic to $\mathbb{C}^n\setminus\Delta$.
\end{proposition}

Note that the space parametrizing \emph{unframed} quadratic differentials up to the action of $PGL_2(\mathbb{C})$ is identified with the quotient of $\mathbb{C}^n\setminus\Delta$ by $\mathbb{Z}/(n+3)\mathbb{Z}$ where the latter acts by changing the framing. Further discussion of the resulting moduli space can be found in Section~12 of~\cite{BridgelandSmith}.

\subsection{Ideal triangulations and the WKB triangulation}

Next we define the WKB triangulation of a quadratic differential. We begin with some general background on ideal triangulations from~\cite{FST}.

\begin{definition}
An \emph{arc} in $(\mathbb{D},\mathbb{M})$ is a smooth path $\gamma$ in $\mathbb{D}$ connecting points of $\mathbb{M}$ whose interior lies in the interior of $\mathbb{D}$ and which has no self-intersections. We also require that $\gamma$ is not homotopic, relative to its endpoints, to a single point or to a path in $\partial\mathbb{D}$ whose interior contains no marked points. Two arcs are considered to be equivalent if they are related by a homotopy through such arcs. A path that connects two marked points and lies entirely on the boundary of~$\mathbb{D}$ without passing through a third marked point is called a \emph{boundary segment}.
\end{definition}

\begin{definition}
Two arcs are said to be \emph{compatible} if there exist curves in their respective homotopy classes that do not intersect in the interior of $\mathbb{D}$. A maximal collection of pairwise compatible arcs is called an \emph{ideal triangulation} of~$(\mathbb{D},\mathbb{M})$. The arcs of an ideal triangulation cut $\mathbb{D}$ into regions called \emph{ideal triangles}. We will write $J=J^T$ for the set of all arcs in an ideal triangulation~$T$.
\end{definition}

\begin{definition}
Let $k\in J$ be an arc in an ideal triangulation $T$. A \emph{flip} at $k$ is the transformation of $T$ that removes $k$ and replaces it with a unique different arc that, together with the remaining arcs, forms a new ideal triangulation:
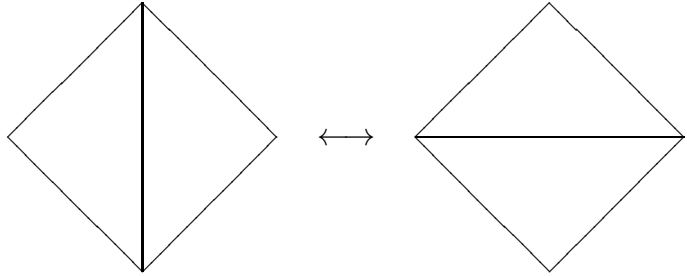
\[
\xy /l1.5pc/:
{\xypolygon4"A"{~:{(2,2):}}},
{\xypolygon4"B"{~:{(2.5,0):}~>{}}},
{\xypolygon4"C"{~:{(0.8,0.8):}~>{}}},
{"A1"\PATH~={**@{-}}'"A3"},
\endxy
\quad
\longleftrightarrow
\quad
\xy /l1.5pc/:
{\xypolygon4"A"{~:{(2,2):}}},
{\xypolygon4"B"{~:{(2.5,0):}~>{}}},
{\xypolygon4"C"{~:{(0.8,0.8):}~>{}}},
{"A2"\PATH~={**@{-}}'"A4"}
\endxy
\]
\end{definition}

The following is a well known fact about ideal triangulations:

\begin{proposition}[\cite{FST}, Proposition~3.8]
Any two ideal triangulations of $(\mathbb{D},\mathbb{M})$ are related by a sequence of flips.
\end{proposition}

Let $(\phi,f)\in\Quad_{fr}(\mathbb{D},\mathbb{M})$ be a framed differential. By definition, $\phi$ is a quadratic differential with simple zeros and a single pole of order $n+5$. If $\phi$ is saddle-free, then Proposition~\ref{prop:horizontalstripdecomposition} says that the separating trajectories divide $\mathbb{CP}^1$ into finitely many horizontal strips and half planes. If we choose a single generic trajectory in each of these regions and map these generic trajectories to~$\mathbb{D}$ using the framing, we get an ideal triangulation of~$(\mathbb{D},\mathbb{M})$.

\begin{definition}
The ideal triangulation of $(\mathbb{D},\mathbb{M})$ obtained by choosing a single generic trajectory in each horizontal strip and half plane of $\phi$ is called the \emph{WKB triangulation}.
\end{definition}

We will return to the WKB triangulation later when we discuss the relationship between the wall-and-chamber structure of the space of stability conditions and the algebraic tori that define the cluster Poisson variety.

\subsection{The spectral cover}

Let $\phi$ be a quadratic differential with simple zeros and poles of order $m_i$ at the points $p_i\in\mathbb{CP}^1$. We can alternatively view $\phi$ as a holomorphic section 
\[
\varphi\in H^0(\mathbb{CP}^1,\omega_{\mathbb{CP}^1}(E)^{\otimes2}), \quad E=\sum_i\left\lceil\frac{m_i}{2}\right\rceil p_i
\]
with simple zeros at both the zeros and odd order poles of~$\phi$.

\begin{definition}
The \emph{spectral cover} is defined as 
\[
\Sigma_\phi=\{(p,\psi(p)):p\in\mathbb{CP}^1,\psi(p)\in L_p,\psi(p)\otimes\psi(p)=\varphi(p)\}\subseteq L.
\]
where $L$ denotes the total space of the line bundle $\omega_{\mathbb{CP}^1}(E)$.
\end{definition}

The projection $\pi:\Sigma_\phi\rightarrow\mathbb{CP}^1$ is a double cover branched at the simple zeros and odd order poles of~$\phi$. Note that the inverse image of the horizontal foliation of $\mathbb{CP}^1\setminus\Crit(\phi)$ under the covering map $\pi$ determines a foliation of $\Sigma_\phi\setminus\pi^{-1}\Crit(\phi)$. Moreover, the following lemma shows that the leaves of this foliation have natural orientations.

\begin{lemma}
Let $\beta$ be a generic trajectory of the quadratic differential $\phi$. Then the real part of the distinguished local coordinate $w(z)$ is increasing or decreasing along~$\beta$.
\end{lemma}

\begin{proof}
Suppose the arc $\beta$ is given by a map $\gamma:(0,1)\rightarrow\mathbb{CP}^1$ with $\gamma'(t)\neq0$ for all $t\in(0,1)$. By definition of a trajectory, we know that the imaginary part $\Im w$ is constant along~$\beta$, that is, $\frac{d}{dt}\Im w(\gamma(t))=0$ for all $t\in(0,1)$. If we also have $\frac{d}{dt}\Re w(\gamma(t))=0$ at some~$t=t_0$, then 
\[
w'(\gamma(t_0))\gamma'(t_0)=\frac{d}{dt}\Re w(\gamma(t))\big\rvert_{t=t_0}+i\frac{d}{dz}\Im w(\gamma(t))\big\rvert_{t=t_0}=0,
\]
so $w'(\gamma(t_0))=0$. But then $\phi(z)=(w'(z))^2dz^{\otimes2}$ has a zero at~$z=\gamma(t_0)$, contradicting the fact that $\beta$ is a generic trajectory. Hence $\frac{d}{dt}\Re w(\gamma(t))\neq0$ for all $t\in(0,1)$, and $\Re w$ is increasing or decreasing along~$\beta$.
\end{proof}

\section{Configurations of points in~$\mathbb{CP}^1$}
\label{sec:ConfigurationsOfPointsInCP1}

\subsection{The moduli space and its cluster structure}

In this section, we will discuss a moduli space parametrizing configurations of points in~$\mathbb{CP}^1$. This moduli space is a special case of Fock and Goncharov's moduli space of framed $PGL_2(\mathbb{C})$-local systems on a surface~\cite{FG1}. As before, we will write $(\mathbb{D},\mathbb{M})$ for a disk with finitely many marked points on its boundary.

\begin{definition}
We will write $\mathcal{X}(\mathbb{D},\mathbb{M})$ for the space of maps $\psi:\mathbb{M}\rightarrow\mathbb{CP}^1$ modulo the action of $PGL_2(\mathbb{C})$ by left multiplication.
\end{definition}

As we recall below, this space $\mathcal{X}(\mathbb{D},\mathbb{M})$ admits an atlas of rational cluster coordinates. To define them, we will first identify the Zariski open subsets on which these coordinates are regular.

\Needspace*{2\baselineskip}
\begin{definition}
Let $\psi$ be a point of $\mathcal{X}(\mathbb{D},\mathbb{M})$.
\begin{enumerate}
\item If $i$ is an arc or boundary segment in~$\mathbb{D}$ and $\psi$ assigns distinct points of~$\mathbb{CP}^1$ to the endpoints of~$i$, then we say that $\psi$ is \emph{generic} with respect to~$i$.

\item If $T$ is an ideal triangulation of~$(\mathbb{D},\mathbb{M})$ and $\psi$ is generic with respect to all edges of~$T$, then we say that $\psi$ is \emph{generic} with respect to~$T$.

\item We say that $\psi$ is \emph{generic} if it is generic with respect to some ideal triangulation of~$(\mathbb{D},\mathbb{M})$.
\end{enumerate}
We will denote the set of generic points of $\mathcal{X}(\mathbb{D},\mathbb{M})$ by $\mathcal{X}^*(\mathbb{D},\mathbb{M})$.
\end{definition}

The set $\mathcal{X}^*(\mathbb{D},\mathbb{M})$ is a Zariski open subset of the stack $\mathcal{X}(\mathbb{D},\mathbb{M})$. It admits a natural atlas of cluster coordinates given by cross ratios. To define these coordinates, let us fix an ideal triangulation~$T$ of~$(\mathbb{D},\mathbb{M})$ and a point $\psi\in\mathcal{X}(\mathbb{D},\mathbb{M})$ which is generic with respect to~$T$. If $j$ is any arc in~$T$, then there are two triangles of $T$ that share this arc~$j$. Together they form a quadrilateral, and we can denote the vertices of this quadrilateral by $p_1,\dots,p_4$ in the counterclockwise direction so that the edge $j$ connects the vertices labeled~$p_1$ and~$p_3$.
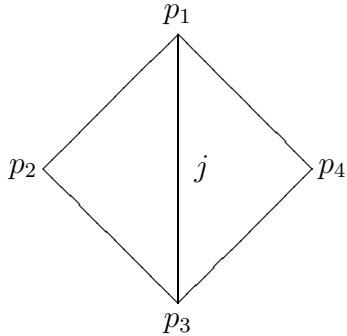
\[
\xy /l1.5pc/:
{\xypolygon4"A"{~:{(2,2):}}},
{\xypolygon4"B"{~:{(2.5,0):}~>{}}},
{\xypolygon4"C"{~:{(0.8,0.8):}~>{}}},
{"A1"\PATH~={**@{-}}'"A3"},
(0.5,0)*{j};
(1,-3.25)*{p_1};
(4.25,0)*{p_2};
(1,3.25)*{p_3};
(-2.25,0)*{p_4};
\endxy
\]
For $i=1,\dots,4$, we set $z_i=\psi(p_i)$ and define $X_j$ as the cross ratio 
\[
X_j=\frac{(z_1-z_2)(z_3-z_4)}{(z_2-z_3)(z_1-z_4)}.
\]
Since we assume $\psi$ is generic with respect to~$T$, this cross ratio is a nonzero complex number. Note that there are two ways of numbering the vertices of the quadrilateral, and they give the same value for the cross ratio.

\begin{proposition}
For any ideal triangulation $T$, the $X_j$ provide an isomorphism from the set of points of $\mathcal{X}(\mathbb{D},\mathbb{M})$ generic with respect to~$T$ to the algebraic torus~$(\mathbb{C}^*)^n$.
\end{proposition}

\begin{proof}
Suppose we are given a nonzero complex number $X_j$ for each arc $j$ of the triangulation~$T$. Let $t_0$ be any triangle of~$T$, and let us assign three arbitrary distinct points of~$\mathbb{CP}^1$ to the vertices of~$t_0$. Next, consider another triangle $t$ in the triangulation which shares an edge $j$ with~$t_0$. We have assigned points of $\mathbb{CP}^1$ to two of its vertices. We can assign a point of~$\mathbb{CP}^1$ to the remaining vertex in such a way the resulting cross ratio is the number~$X_j$. Continuing in this way, we assign a point of $\mathbb{CP}^1$ to each element of~$\mathbb{M}$. It is easy to see that this construction defines a regular map from~$(\mathbb{C}^*)^n$ to the set of points in $\mathcal{X}(\mathbb{D},\mathbb{M})$ which are generic with respect to~$T$. It is a two-sided inverse of the map sending such a point to its coordinates.
\end{proof}

If $T'$ is the ideal triangulation obtained from~$T$ by performing a flip at some arc~$k\in J$, then the set of arcs $J=J^T$ is naturally in bijection with the set of arcs $J'=J^{T'}$, and we can use the construction described above to associate a number $X_j'\in\mathbb{C}^*$ to each arc $j\in J'=J$. The following proposition relates these to the numbers $X_j$ computed using the ideal triangulation~$T$.

\begin{proposition}[\cite{FG1}, Section~10]
\label{prop:transformcoordinates}
The numbers $X_j'$ are given in terms of the numbers $X_j$~($j\in J$) by the following graphical rule:
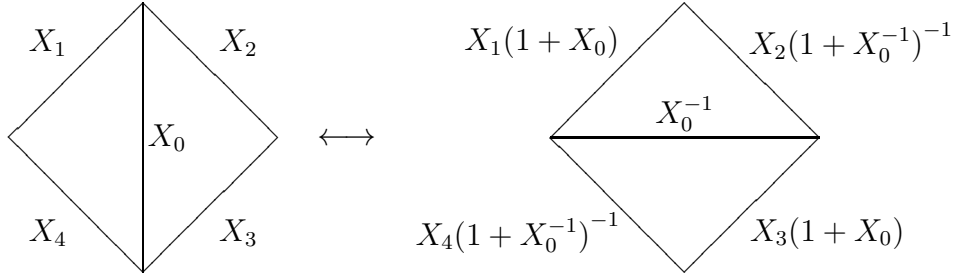
\[
\xy /l1.5pc/:
{\xypolygon4"A"{~:{(2,2):}}},
{\xypolygon4"B"{~:{(2.5,0):}~>{}}},
{\xypolygon4"C"{~:{(0.8,0.8):}~>{}}},
{"A1"\PATH~={**@{-}}'"A3"},
(.5,0)*{X_0},
(-1,2)*{X_3},
(3,2)*{X_4},
(-1,-2)*{X_2},
(3,-2)*{X_1}
\endxy
\quad
\longleftrightarrow
\quad
\xy /l1.5pc/:
{\xypolygon4"A"{~:{(2,2):}}},
{\xypolygon4"B"{~:{(2.5,0):}~>{}}},
{\xypolygon4"C"{~:{(0.8,0.8):}~>{}}},
{"A2"\PATH~={**@{-}}'"A4"},
(1,-0.5)*{X_0^{-1}},
(-2,2)*{X_3(1+X_0)},
(4.5,2)*{X_4{(1+X_0^{-1})}^{-1}},
(-2.5,-2)*{X_2{(1+X_0^{-1})}^{-1}},
(4,-2)*{X_1(1+X_0)}
\endxy
\]
\end{proposition}

The transformation depicted in these diagrams is an example of a cluster transformation. In Section~\ref{sec:StabilityConditionsAndTheClusterVariety}, we will use the same rule to give an abstract definition of the cluster Poisson variety.

\subsection{Generic configurations of points}

We conclude this section by giving an alternative characterization of the set $\mathcal{X}^*(\mathbb{D},\mathbb{M})$ of generic points in~$\mathcal{X}(\mathbb{D},\mathbb{M})$. This will allow us to prove in Section~\ref{sec:TheMainConstruction} that the set of generic points parametrizes the monodromy data of certain differential equations.

\begin{proposition}
\label{prop:characterizegeneric}
Denote the points of $\mathbb{M}$ by $p_1,\dots,p_{n+3}$ so that the order of the indices is compatible with the orientation of~$\mathbb{D}$. Then the set $\mathcal{X}^*(\mathbb{D},\mathbb{M})$ consists of points $\psi$ in $\mathcal{X}(\mathbb{D},\mathbb{M})$ satisfying the following two conditions:
\begin{enumerate}
\item We have $\psi(p_i)\neq\psi(p_{i+1})$ for all $i$ where we consider the indices modulo $n+3$.
\item The set $\{\psi(p_1),\dots,\psi(p_{n+3})\}$ contains at least three distinct points.
\end{enumerate}
\end{proposition}

\begin{proof}
If $\psi$ is a generic point of $\mathcal{X}(\mathbb{D},\mathbb{M})$, then it is generic with respect to some ideal triangulation~$T$. Therefore the points of~$\mathbb{CP}^1$ that $\psi$ associates to the ends of any edge $i$ of~$T$ are distinct. Thus $\psi$ satisfies the first condition. Since $\psi$ assigns three distinct points of~$\mathbb{CP}^1$ to the vertices of any triangle, $\psi$ also satisfies the second condition.

Conversely, suppose $\psi$ is a point of $\mathcal{X}(\mathbb{D},\mathbb{M})$ which satisfies the conditions in the proposition. We need to show that there exists an ideal triangulation $T$ of $(\mathbb{D},\mathbb{M})$ such that $\psi$ is generic with respect to~$T$. If $|\mathbb{M}|=3$, this follows immediately from our assumptions on~$\psi$. If $|\mathbb{M}|>3$, then we can find some $k$ such that $\psi(p_{k-1})\neq\psi(p_{k+1})$ where the indices are considered modulo $n+3$. Consider the set 
\[
\mathcal{Z}=\{\psi(p_i):i\neq k\}.
\]
If this set contains at least three distinct points, draw an arc connecting $p_{k-1}$ and $p_{k+1}$. By induction on~$|\mathbb{M}|$, this can be extended to an ideal triangulation $T$ such that $\psi$ is generic with respect to~$T$. On the other hand, if $\mathcal{Z}$ contains only two distinct points, then these points are distinct from $\psi(p_k)$. Therefore, if $T$ is the ideal triangulation containing an arc from $p_k$ to each point $\mathbb{M}\setminus\{p_k\}$, then $\psi$ is generic with respect to~$T$. This completes the proof.
\end{proof}

\section{The main construction}
\label{sec:TheMainConstruction}

\subsection{From quadratic differentials to monodromy data}

In Section~\ref{sec:QuadraticDifferentials}, we defined the moduli space $\Quad_{fr}(\mathbb{D},\mathbb{M})$ of framed quadratic differentials and proved that it is identified with a space of polynomials of the form 
\[
P(z)=z^{n+1}+a_{n-1}z^{n-1}+\dots+a_1z+a_0
\]
for $a_0,\dots,a_{n-1}\in\mathbb{C}$. We will now consider the associated differential equation 
\begin{align}
\label{eqn:schrodinger}
y''(z)-P(z)y(z)=0.
\end{align}
This differential equation was studied by Sibuya~\cite{Sibuya}, who worked out the detailed asymptotic behavior of solutions as $z\rightarrow\infty$. Because this equation has an irregular singularity, this asymptotic behavior depends on the direction in which we approach~$\infty$.

\begin{definition}
For $k\in\mathbb{Z}/(n+3)\mathbb{Z}$, the $k$th \emph{Stokes sector} is defined as the set 
\[
\mathscr{S}_k=\left\{z\in\mathbb{C}:\left|\arg z-\frac{2\pi}{n+3}k\right|<\frac{\pi}{n+3}\right\}.
\]
A solution $y(z)$ of~\eqref{eqn:schrodinger} is said to be \emph{subdominant} in~$\mathscr{S}_k$ if $y(z)\rightarrow0$ as $z\rightarrow\infty$ along any ray in~$\mathscr{S}_k$.
\end{definition}

The Stokes sectors defined in this way are non-overlapping regions whose closures cover the complex plane entirely. As shown in Chapter~2 of~\cite{Sibuya}, there is a collection of solutions $Y_k(z)=Y_k(z,a_1,\dots,a_n)$ of~\eqref{eqn:schrodinger} such that $Y_k(z)$ is subdominant in the $k$th Stokes sector.

In Chapter~8 of~\cite{Sibuya}, Sibuya discusses a general boundary value problem related to~\eqref{eqn:schrodinger}. The problem is formulated in terms of certain limits in the Stokes sectors, which we now define.

\begin{definition}
Let $y_1(z)$ and $y_2(z)$ be linearly independent solutions of~\eqref{eqn:schrodinger}. For any $k\in\mathbb{Z}/(n+3)\mathbb{Z}$, we call 
\[
w_k(y_1,y_2)=\lim_{\substack{z\rightarrow\infty \\ z\in\mathscr{S}_k}}\frac{y_1(z)}{y_2(z)}
\]
the $k$th \emph{asymptotic value}. Here the limit is taken along any ray that lies in the $k$th Stokes sector.
\end{definition}

By expressing the above limit in terms of subdominant solutions of~\eqref{eqn:schrodinger}, one shows that this limit exists in~$\mathbb{CP}^1$ and is independent of the direction in~$\mathscr{S}_k$ in which the limit is taken (see~\cite{Sibuya}, Chapter~8). Of course the asymptotic values do depend on the choice of~$y_1$ and~$y_2$, but the following proposition shows that the collection of asymptotic values is well defined up to the action of~$PGL_2(\mathbb{C})$.

\begin{proposition}
\label{prop:PGL2action}
If $\tilde{y}_1$ and $\tilde{y}_2$ are linearly independent solutions of~\eqref{eqn:schrodinger}, then there exists a transformation $g\in PGL_2(\mathbb{C})$ so that 
\[
w_k(\tilde{y}_1,\tilde{y}_2)=g\cdot w_k(y_1,y_2)
\]
for all $k\in\mathbb{Z}/(n+3)\mathbb{Z}$.
\end{proposition}

\begin{proof}
If $\tilde{y}_1$ and $\tilde{y}_2$ are linearly independent solutions of~\eqref{eqn:schrodinger} then we can write $\tilde{y}_1=ay_1+by_2$ and $\tilde{y}_2=cy_1+dy_2$ for some $\left(\begin{array}{cc} a & b \\ c & d \end{array}\right)\in GL_2(\mathbb{C})$ since the solutions $y_1$ and~$y_2$ form a basis. Then 
\[
w_k(\tilde{y}_1,\tilde{y}_2) = \lim_{\substack{z\rightarrow\infty \\ z\in\mathscr{S}_k}}\frac{ay_1(z)+by_2(z)}{cy_1(z)+dy_2(z)} = \frac{aw_k(y_1,y_2)+b}{cw_k(y_1,y_2)+d}
\]
as desired.
\end{proof}

Consider the framed quadratic differential $(\phi,f_0)$ corresponding to the polynomial $P(z)$. The quadratic differential $\phi$ is given by $\phi(z)=P(z)dz^{\otimes2}$. The $n+3$ distinguished tangent directions at~$\infty$ are the tangent vectors to the rays
\[
\mathbb{R}_{>0}\exp\left(2\pi i\frac{k}{n+3}\right), 
\]
which lie inside the Stokes sectors, and the framing $f_0$ provides a bijection between these distinguished tangent directions and the points of~$\mathbb{M}$. Thus we can define a map $\psi:\mathbb{M}\rightarrow\mathbb{CP}^1$ which assigns to each point of~$\mathbb{M}$ the asymptotic value in the corresponding Stokes sector. By~Proposition~\ref{prop:PGL2action}, this construction provides a well defined map 
\[
\mathbb{C}^n\rightarrow\mathcal{X}(\mathbb{D},\mathbb{M})
\]
sending the vector $(a_0,a_1,\dots,a_{n-1})$ to the point~$\psi$. Moreover, we have the following fact.

\begin{proposition}[\cite{Sibuya}, Lemma~39.1 and Lemma~39.2]
Let $w_k=w_k(y_1,y_2)$ be the asymptotic values associated with linearly independent solutions $y_1(z)$ and $y_2(z)$ of~\eqref{eqn:schrodinger}. Then 
\begin{enumerate}
\item We have $w_k\neq w_{k+1}$ for all $k\in\mathbb{Z}/(n+3)\mathbb{Z}$.
\item The set $\{w_1,\dots,w_{n+3}\}$ contains at least three distinct elements.
\end{enumerate}
\end{proposition}

Thus it follows from Proposition~\ref{prop:characterizegeneric} that we have a map 
\[
H:\mathbb{C}^n\rightarrow\mathcal{X}^*(\mathbb{D},\mathbb{M})
\]
to the set of generic points of $\mathcal{X}(\mathbb{D},\mathbb{M})$. Our next result follows from work of Bakken~\cite{Bakken} (reviewed in~\cite{Sibuya}):

\begin{theorem}
The space $\mathcal{X}^*(\mathbb{D},\mathbb{M})$ has the structure of a possibly non-Hausdorff complex manifold, and the map $H$ is a local biholomorphism.
\end{theorem}

\begin{proof}
Denote the points of $\mathbb{M}$ by $p_1,\dots,p_{n+3}$ so that the order of the indices is compatible with the orientation of~$\mathbb{D}$. Considering these indices modulo $n+3$, we define, for each $k=1,\dots,n+3$, the set 
\[
U_k=\{\psi\in\mathcal{X}^*(\mathbb{D},\mathbb{M}):\text{$\psi(p_k)$, $\psi(p_{k+1})$, and $\psi(p_{k+2})$ are distinct}\}.
\]
It follows from conditions~1 and~2 of Proposition~\ref{prop:characterizegeneric} that for any $\psi\in\mathcal{X}^*(\mathbb{D},\mathbb{M})$, we can find an index $k$ such that $\psi(p_k)$, $\psi(p_{k+1})$, and $\psi(p_{k+2})$ are distinct. Hence $\{U_k\}$ is an open cover of $\mathcal{X}^*(\mathbb{D},\mathbb{M})$. Suppose we are given a point $\psi\in U_1$. Then we can find an element $g\in PGL_2(\mathbb{C})$ such that 
\[
g\psi(p_1)=0, \quad g\psi(p_2)=1, \quad g\psi(p_3)=\infty.
\]
Define 
\[
F_1(\psi)=(g\psi(p_4),\dots,g\psi(p_{n+3})).
\]
In this way, we get a map $F_1:U_1\rightarrow V_1$ where 
\[
V_1=\{(z_4,\dots,z_{n+3})\in(\mathbb{CP}^1)^n:z_i\neq z_{i+1},z_4\neq\infty,z_{n+3}\neq0\}.
\]
This map is a homeomorphism with inverse given by 
\[
F_1^{-1}(z_4,\dots,z_{n+3})=(0,1,\infty,z_4,\dots,z_{n+3}).
\]
In a similar way, we define a manifold $V_k\subseteq(\mathbb{CP}^1)^n$ and a homeomorphisms $F_k:U_k\rightarrow V_k$ for all~$k$. The transition maps are holomorphic, giving $\mathcal{X}^*(\mathbb{D},\mathbb{M})$ the structure of a complex manifold. The main result of~\cite{Bakken} implies that the map $H$ defined above is a local biholomorphism into this manifold.
\end{proof}

Restricting $H$ to the complement of the discriminant locus in~$\mathbb{C}^n$, we obtain a local biholomorphism 
\[
H:\Quad_{fr}(\mathbb{D},\mathbb{M})\rightarrow\mathcal{X}^*(\mathbb{D},\mathbb{M}).
\]
It is easy to see that this map is $\mathbb{Z}/(n+3)\mathbb{Z}$-equivariant where a generator of this group acts on a differential $\phi(z)=P(z)dz^{\otimes2}$ by multiplying $z$ by a primitive $(n+3)$rd root of unity and acts on a map $\psi$ by cyclically permuting its values.

\subsection{Definition of the WKB solutions}

To understand the properties of the map $H$, we will use results from the theory of exact WKB analysis. This theory allows us to construct solutions of the differential equation 
\begin{align}
\label{eqn:schrodingerwithparameter}
\hbar^2\frac{d^2}{dz^2}y(z,\hbar)-\varphi(z)y(z,\hbar)=0
\end{align}
where $\varphi(z)$ is a meromorphic function and $\hbar$ is a small parameter. We do this by taking the Borel sums of certain formal power series in the parameter~$\hbar$. Let us review the basic construction, following~\cite{IwakiNakanishi}.

To define the formal power series, one begins with a family $\{S_k(z)\}_{k\geq-1}$ of functions satisfying the initial condition 
\[
S_{-1}^2=\varphi(z)
\]
and the recursion relations 
\[
2S_{-1}S_{k+1}+\sum_{\substack{k_1+k_2=k \\ 0\leq k_j\leq k}}S_{k_1}S_{k_2}+\frac{dS_k}{dz}=0
\]
for $k\geq-1$. There are two families $\{S_k^{(+)}(z)\}_{k\geq-1}$ and $\{S_k^{(-)}(z)\}_{k\geq-1}$ of functions satisfying these recursion relations depending on the choice of root $S_{-1}^{(\pm)}(z)=\pm\sqrt{\varphi(z)}$ for the initial condition. We consider the formal series 
\[
S^{(\pm)}(z,\hbar)=\sum_{k=-1}^\infty\hbar^kS_k^{(\pm)}(z)
\]
and the ``odd part'' 
\[
S_{\text{odd}}(z,\hbar)=\frac{1}{2}\left(S^{(+)}(z,\hbar)-S^{(-)}(z,\hbar)\right).
\]
We use this expression to get a pair of formal solutions of~\eqref{eqn:schrodingerwithparameter}.

\begin{definition}[\cite{IwakiNakanishi}, Definition~2.9]
\label{def:WKBsolutions}
The \emph{WKB solutions} are the formal solutions 
\[
\psi_{\pm}(z,\hbar)=\frac{1}{\sqrt{S_{\text{odd}}(z,\hbar)}}\exp\left(\pm\int^zS_{\text{odd}}(\zeta,\hbar)d\zeta\right)
\]
of equation~\eqref{eqn:schrodingerwithparameter}.
\end{definition}

The function $\varphi(z)$ in~\eqref{eqn:schrodingerwithparameter} determines a quadratic differential $\phi$ given in a local coordinate by $\phi(z)=\varphi(z)dz^{\otimes2}$. We denote by $\Sigma_\phi$ the spectral cover associated with this quadratic differential. The integral appearing in Definition~\ref{def:WKBsolutions} is defined by integrating the coefficient of each power of $\hbar$ in the formal series $S_{\text{odd}}(\zeta,\hbar)$ termwise. Since the coefficients of $S_{\text{odd}}(\zeta,\hbar)$ are multivalued functions on $\mathbb{CP}^1\setminus\Crit(\phi)$, the path of the integral should be considered in the Riemann surface~$\Sigma_\phi$. To get a well defined formal power series, we must also specify the lower limit of integration. We would like to take this lower limit to be a pole of~$\phi$, but we are unable to do so because the 1-form $S_{\text{odd}}(z,\hbar)dz$ is not integrable near a pole. Thus we introduce the regularized expression 
\[
S_{\text{odd}}^{\text{reg}}(z,\hbar)dz=\left(S_{\text{odd}}(z,\hbar)-\frac{1}{\hbar}\sqrt{\varphi(z)}\right)dz.
\]
From now on, we will assume the potential $\varphi(z)$ satisfies Assumptions~2.3 and~2.5 of~\cite{IwakiNakanishi}. These assumptions will always be satisfied in the examples we consider, where the potential is given by a nonconstant polynomial with simple zeros.

\begin{theorem}[\cite{IwakiNakanishi}, Proposition~2.8]
If $\phi$ satisfies Assumption~2.5 of~\cite{IwakiNakanishi}, then the formal series valued 1-form $S_{\text{odd}}^{\text{reg}}(z,\eta)dz$ is integrable at every pole of~$\phi$.
\end{theorem}

Thus we can employ following scheme to define the integral in Definition~\ref{def:WKBsolutions}.

\begin{definition}
\label{def:normalizedWKB}
Let $p$ be a pole of~$\phi$. Then the WKB solution \emph{normalized at $p$} is the expression 
\[
\psi_{\pm}(z,\hbar)=\frac{1}{\sqrt{S_{\text{odd}}(z,\hbar)}}\exp\left(\pm\left(\frac{1}{\hbar}\int_a^z\sqrt{\varphi(\zeta)}d\zeta+\int_p^zS_{\text{odd}}^{\text{reg}}(\zeta,\hbar)d\zeta\right)\right)
\]
where $a$ is any zero of $\phi$ independent of~$p$. The first integral is taken along a path $\gamma_a$ in~$\Sigma_\phi$, and the second integral is taken along a path~$\gamma_p$ in~$\Sigma_\phi$. We will write $\psi_{\pm}^{(\gamma_p,\gamma_a)}(z,\hbar)$ for the WKB solution normalized at~$p$ when we wish to emphasize the choice of~$\gamma_p$ and~$\gamma_a$.
\end{definition}

Thus we have a pair of well-normalized formal solutions of~\eqref{eqn:schrodingerwithparameter}. Each WKB solution can be expanded as a formal power series in $\hbar$ multiplied by an additional factor:
\[
\psi_{\pm}(z,\hbar)=\exp\left(\pm\frac{1}{\hbar}\int_a^z\sqrt{\varphi(\zeta)}d\zeta\right)\hbar^{1/2}\sum_{k=0}^\infty\hbar^k\psi_{\pm,k}(z).
\]
It is known that the series in this expression is divergent in general, and therefore, in order to get a genuine analytic solution of~\eqref{eqn:schrodingerwithparameter}, we must take the Borel resummation. By definition, an expression of the form $f(\eta)=e^{\eta s}\eta^{-\rho}\sum_{k=0}^\infty\eta^{-k}f_n$ is \emph{Borel summable} if the formal power series $g(\eta)=\sum_{k=0}^\infty\eta^{-k}f_n$ is Borel summable. In this case, the \emph{Borel sum} of $f(\eta)$ is defined as $\mathcal{S}[f](\eta)=e^{\eta s}\eta^{-\rho}\mathcal{S}[g](\eta)$ where $\mathcal{S}[g]$ is the Borel sum of~$g$.

\begin{theorem}[\cite{IwakiNakanishi}, Corollary~2.21]
\label{thm:Borelsummability}
Suppose the differential $\phi$ has at most one saddle trajectory and satisfies Assumptions~2.3 and~2.5 of~\cite{IwakiNakanishi}.
\begin{enumerate}
\item Let $\beta$ be a path in the spectral cover that projects to a generic trajectory of~$\phi$. Then the formal power series $\int_\beta S_{\text{odd}}^{\text{reg}}(z,\hbar)dz$ is Borel summable.
\item Let $D$ be a horizontal strip or half plane defined by~$\phi$. Then a WKB solution, normalized as in Definition~\ref{def:normalizedWKB}, is Borel summable at each point in~$D$. The Borel sum is an analytic solution of~\eqref{eqn:schrodingerwithparameter} on~$D$, which is also analytic in $\eta=\hbar^{-1}$ in a domain $\{\eta\in\mathbb{C}:|\arg\eta-\theta|<\pi/2, |\eta|\gg1\}$ and depends on the choice of paths $\gamma_a$ and~$\gamma_p$.
\end{enumerate}
\end{theorem}

Theorem~\ref{thm:Borelsummability} follows from unpublished work of Koike and Sch\"afke~\cite{KoikeSchafke}, which establishes the Borel summability of formal solutions of the Riccati equation. For a sketch of the proof of their result in the context relevant to the present paper, see~\cite{Takei}.

\subsection{Genericity with respect to the WKB triangulation}

In this subsection, we prove a version of Theorem~\ref{thm:introchambertotorus} from the introduction. As a first step, we examine the asymptotic behavior of the Borel sums of WKB solutions along a generic trajectory. As we will see, this asymptotic behavior depends on the direction in which the real part of the distinguished local coordinate is increasing.

Let $D$ be a horizontal strip or half plane determined by the differential $\phi$, and let $\beta$ be a generic trajectory in~$D$ connecting poles $p_1$ and $p_2$ on the boundary of~$D$. Let $a$ be a zero on the boundary of~$D$, and for any point~$z\in D$, choose paths $\gamma_{p_1}$, $\gamma_{p_2}$ and $\gamma_a$ in some sheet of $\pi^{-1}D$ whose projections connect the points $p_1$, $p_2$, and~$a$, respectively, to~$z$. An example is illustrated below where $D$ is a horizontal strip.
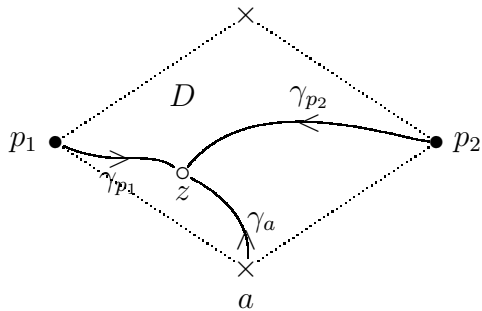
\[
\xy /l2pc/:
(0,-2)*{\times}="21";
(3,0)*{\bullet}="12";
(1,0.5)*{\circ}="22";
(-3,0)*{\bullet}="42";
(0,2)*{\times}="23";
"12";"22" **\crv{(2,0.5) & (1.5,0)};
"42";"22" **\crv{(-2.5,0) & (0,-1)};
"23";"22" **\crv{(-0.25,1) & (1,0.5)};
{"12"\PATH~={**@{.}}'"21"},
{"21"\PATH~={**@{.}}'"42"},
{"12"\PATH~={**@{.}}'"23"},
{"23"\PATH~={**@{.}}'"42"},
(-1,-0.31)*{<},
(-1,-0.7)*{\gamma_{p_2}},
(2,0.26)*{>},
(2,0.65)*{\gamma_{p_1}},
(0,1.55)*{\wedge},
(-0.25,1.25)*{\gamma_a},
(1,0.8)*{z},
(3.5,0)*{p_1},
(-3.5,0)*{p_2},
(0,2.5)*{a},
(1,-0.75)*{D},
\endxy
\]
In the following, we will write $\Psi_{D,\pm}^{(p,a)}=\mathcal{S}[\psi_{\pm}^{(\gamma_p,\gamma_a)}]$.

\begin{lemma}
\label{lem:canonicalcoordinatesubdominant}
Suppose the real part of $w(z)=\int_a^z\sqrt{\varphi(\zeta)}d\zeta$ is increasing as $z$ goes from~$p_1$ to~$p_2$ along~$\beta$. Then $\Psi_{D,+}^{(p_1,a)}(z)$ is decreasing as $z\rightarrow p_1$ along~$\beta$, and $\Psi_{D,-}^{(p_2,a)}(z)$ is decreasing as $z\rightarrow p_2$ along~$\beta$.
\end{lemma}

\begin{proof}
By definition of the Borel resummation, we have 
\[
\Psi_{D,+}^{(p_1,a)}(z)=\exp\left(\frac{1}{\hbar}\int_{\gamma_a}\sqrt{\varphi(\zeta)}d\zeta\right)\mathcal{S}[g]
\]
where 
\[
g(z,\hbar)=\frac{1}{\sqrt{S_{\text{odd}}(z,\hbar)}}\exp\left(\int_{\gamma_{p_1}}S_{\text{odd}}^{\text{reg}}(\zeta,\hbar)d\zeta\right).
\]
The function $\mathcal{S}[g]$ is bounded as $z\rightarrow p_1$ along~$\beta$. Under our assumptions, the exponential factor in the expression for $\Psi_{D,+}^{(p_1,a)}$ is decreasing, so the function $\Psi_{D,+}^{(p_1,a)}$ is decreasing. The second statement is proved by a similar argument.
\end{proof}

\begin{lemma}
\label{lem:horizontalstrip}
We have 
\[
\Psi_{D,+}^{(p_1,a)}=x_{12}\Psi_{D,+}^{(p_2,a)} \quad \text{and} \quad \Psi_{D,-}^{(p_1,a)}=x_{12}^{-1}\Psi_{D,-}^{(p_2,a)}
\]
where 
\[
x_{12}=\mathcal{S}\left[\exp\left(\int_{p_1}^{p_2}S_{\text{odd}}^{\text{reg}}(z,\hbar)dz\right)\right].
\]
\end{lemma}

\begin{proof}
By the properties of Borel resummation, we have 
\begin{align*}
x_{12}\Psi_{D,+}^{(p_2,a)} &= \exp\left(\frac{1}{\hbar}\int_{\gamma_a}\sqrt{\varphi(\zeta)}d\zeta\right)\mathcal{S}\biggl[\frac{1}{\sqrt{S_{\text{odd}}(z,\hbar)}}\exp\left(\int_{\gamma_{p_2}}S_{\text{odd}}^{\text{reg}}(\zeta,\hbar)d\zeta\right) \\
& \qquad \cdot\exp\left(\int_{p_1}^{p_2}S_{\text{odd}}^{\text{reg}}(\zeta,\hbar)d\zeta\right)\biggr] \\
&= \exp\left(\frac{1}{\hbar}\int_{\gamma_a}\sqrt{\varphi(\zeta)}d\zeta\right)\mathcal{S}\left[\frac{1}{\sqrt{S_{\text{odd}}(z,\hbar)}}\exp\left(\int_{\gamma_{p_1}}S_{\text{odd}}^{\text{reg}}(\zeta,\hbar)d\zeta\right)\right] \\
&= \Psi_{D,+}^{(p_1,a)}
\end{align*}
and similarly for the other identity.
\end{proof}

There is a natural $\mathbb{C}^*$-action on the space $\mathbb{C}^n\setminus\Delta$ considered above. For any $t\in\mathbb{C}^*$ and any $(a_0,\dots,a_{n-1})\in\mathbb{C}^n\setminus\Delta$, one has 
\[
t\cdot(a_0,a_1,\dots,a_{n-1})=(t^{n+1}a_0,t^na_1,\dots,t^2a_{n-1}).
\]
Using the isomorphism of Proposition~\ref{prop:identificationquad}, we can view this as a $\mathbb{C}^*$-action on the space of framed differentials $\Quad_{fr}(\mathbb{D},\mathbb{M})$. If $(\phi,f_0)\in\Quad_{fr}(\mathbb{D},\mathbb{M})$ is a framed differential given by $\phi(z)=P(z)dz^{\otimes}$ for some monic polynomial $P(z)$ as in Subsection~\ref{sec:TheModuliSpaceOfFramedDifferentials}, then an element $t\in\mathbb{C}^*$ acts on $(\phi,f_0)$ by modifying the coefficients of $P$ and leaving the framing $f_0$ unchanged. Consider, for each $\epsilon>0$, the region 
\[
\mathbb{H}(\epsilon)=\{\hbar\in\mathbb{C}:|\hbar|<\epsilon \text{ and } \Re(\hbar)>0\}.
\]
For any $\hbar\in\mathbb{H}(\epsilon)$, we define $H_\hbar(\phi,f)=H(t\cdot(\phi,f))$ where $t=\hbar^{-2/(n+3)}$ is defined using the principal branch of the logarithm. Thus, for any $\hbar\in\mathbb{H}(\epsilon)$, we obtain a map 
\[
H_\hbar:\Quad_{fr}(\mathbb{D},\mathbb{M})\rightarrow\mathcal{X}^*(\mathbb{D},\mathbb{M}).
\]
The following is a version of Theorem~\ref{thm:introchambertotorus}.

\begin{theorem}
\label{thm:chambertotorus}
Suppose $(\phi,f)$ is a saddle-free differential in $\Quad_{fr}(\mathbb{D},\mathbb{M})$. Then there exists $\epsilon>0$ such that for all points $\hbar\in\mathbb{H}(\epsilon)$, the point $H_\hbar(\phi,f)$ is generic with respect to the WKB triangulation of~$(\phi,f)$.
\end{theorem}

\begin{proof}
Let $(\phi,f_0)\in\Quad_{fr}(\mathbb{D},\mathbb{M})$ be a saddle-free differential where $\phi(z)=P(z)dz^{\otimes2}$ is as in Subsection~\ref{sec:TheModuliSpaceOfFramedDifferentials}. Let $\beta$ be an edge of the WKB triangulation. Using the framing $f_0$, we can identify $\beta$ with a generic horizontal trajectory of the foliation of~$\mathbb{CP}^1$ defined by~$\phi$. By Lemma~\ref{lem:canonicalcoordinatesubdominant}, there exists $\epsilon>0$ such that, for any $\hbar\in\mathbb{H}(\epsilon)$, we can find a solution of 
\begin{align}
\label{eqn:schrodingerwithparameterspecialized}
y''(z)-\frac{1}{\hbar^2}P(z)y(z)=0
\end{align}
near each endpoint of $\beta$ which is decreasing as $z$ approaches the endpoint along~$\beta$. Moreover, Lemma~\ref{lem:horizontalstrip} shows that the analytic continuations of these solutions are linearly independent. If we take the coordinate transformation $\widetilde{z}=tz$, then~\eqref{eqn:schrodingerwithparameterspecialized} is equivalent to the equation 
\begin{align}
\label{eqn:transformedschrodinger}
\widetilde{y}''(\widetilde{z})-\widetilde{P}(\widetilde{z})\widetilde{y}(\widetilde{z})=0.
\end{align}
Here $\widetilde{P}$ is the polynomial obtained by modifying the coefficients of~$P$ by the action of~$t$ and we put $\widetilde{y}(\widetilde{z})=y(z)$. The ends of~$\beta$ lie in two Stokes sectors for the equation~\eqref{eqn:transformedschrodinger}. Call these Stokes sectors $\mathscr{S}_p$ and~$\mathscr{S}_q$. Let $Y_p$ be a solution of~\eqref{eqn:schrodingerwithparameterspecialized} which is subdominant in~$\mathscr{S}_p$ and $Y_q$ a solution which is subdominant in~$\mathscr{S}_q$. Since the space of solutions that are subdominant in a given Stokes sector is one-dimensional, it follows that~$Y_p$ and~$Y_q$ are linearly independent. By Lemma~4.1 of~\cite{Bakken}, the asymptotic values associated to the Stokes sectors $\mathscr{S}_p$ and~$\mathscr{S}_q$ are distinct. Hence $H_\hbar(\phi,f_0)$ is generic with respect to the WKB triangulation of~$(\phi,f_0)$.
\end{proof}

\section{Stability conditions and the cluster variety}
\label{sec:StabilityConditionsAndTheClusterVariety}

\subsection{Quivers with potentials and their mutations}

Finally, let us reinterpret the main construction of this paper from a more abstract point of view. We will begin by describing how the combinatorics of ideal triangulations can be encoded in quivers with potentials.

Suppose that $T$ is an ideal triangulation of $(\mathbb{D},\mathbb{M})$. We associate a quiver $Q(T)$ to this triangulation by drawing a single vertex in the interior of each arc of~$T$. If two arcs belong to the same triangle, we connect the vertices by an arrow oriented in the clockwise direction with respect to their common vertex. An example of an ideal triangulation and the associated quiver is illustrated below.
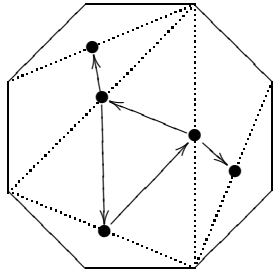
\[
\xy /l1.5pc/:
{\xypolygon8"A"{~:{(3,0):}}},
{"A1"\PATH~={**@{.}}'"A3"'},
{"A1"\PATH~={**@{.}}'"A6"'},
{"A3"\PATH~={**@{.}}'"A6"'},
{"A3"\PATH~={**@{.}}'"A5"'},
{"A8"\PATH~={**@{.}}'"A6"'},
(-0.15,0)*{\bullet}="a", 
(1.8,-0.8)*{\bullet}="b",
(1.75,2)*{\bullet}="c",
(2,-1.85)*{\bullet}="d",
(-1,0.75)*{\bullet}="e", 
\xygraph{
"a":"b",
"b":"c",
"c":"a",
"b":"d",
"a":"e",
}
\endxy
\]
Note that if every triangle of~$T$ contains a boundary segment, then the underlying graph of~$Q(T)$ is the $A_n$ Dynkin diagram.

Recall that a potential for a quiver $Q$ is defined as a linear combination of oriented cycles of~$Q$. We refer to~\cite{DWZ} for the general theory of a quiver with potential. Here we work with a canonical potential for the quiver~$Q(T)$. In the definition that follows, a triangle of~$T$ will be called \emph{internal} if all of its edges are arcs. Note that each internal triangle $t$ contains a 3-cycle $C(t)$ of the quiver~$Q(T)$, oriented in the counterclockwise direction. The canonical potential for~$Q(T)$ is given by 
\[
W(T)=\sum_tC(t)
\]
where the sum runs over all internal triangles of~$T$. This potential is special case of the one introduced by Labardini-Fragoso for general triangulated surfaces~\cite{L}.

An important idea in the theory of cluster algebras is the notion of quiver mutation. It is defined for a finite quiver without loops or oriented 2-cycles. Such a quiver is said to be \emph{2-acyclic}.

\begin{definition}
\label{def:quivermutation}
Let $k$ be a vertex of a 2-acyclic quiver $Q$. Then we define a new quiver $\mu_k(Q)$, called the quiver obtained by \emph{mutation} in the direction~$k$, as follows.
\begin{enumerate}
\item For each pair of arrows $i\rightarrow k\rightarrow j$, introduce a new arrow $i\rightarrow j$.
\item Reverse all arrows incident to~$k$.
\item Remove the arrows from a maximal set of pairwise disjoint 2-cycles.
\end{enumerate}
The last item means, for example, that we replace the diagram $\xymatrix{i \ar@< 4pt>[r] \ar@< -4pt>[r] & j \ar[l]}$ by $\xymatrix{i \ar[r] & j}$.
\end{definition}

In~\cite{DWZ}, Derksen, Weyman, and~Zelevinsky define a similar notion of mutation for a quiver with potential. To do this, they first impose an equivalence relation on the set of quivers with potentials called right equivalence. Roughly speaking, two potentials are right equivalent if they are related by an isomorphism of path algebras that fixes the zero length paths. A potential is said to be \emph{reduced} if it is a sum of cycles of length $\geq3$. If $(Q,W)$ is any reduced quiver with potential and $k$ is a vertex of~$Q$ which is not contained in an oriented 2-cycle, then Derksen, Weyman, and~Zelevinsky define a reduced quiver with potential $\mu_k(Q,W)$ called the quiver with potential obtained by mutation in the direction~$k$. It is well defined up to right equivalence and depends only on the right equivalence class of $(Q,W)$. We refer to~\cite{DWZ} for the detailed definitions.

There is a subtle point concerning the definition of mutation in~\cite{DWZ}. If $(Q,W)$ is a 2-acyclic quiver with potential, then $\mu_k(Q,W)$ exists for any~$k$, but it may not be 2-acyclic. We call $(Q,W)$ \emph{nondegenerate} if the quiver with potential obtained by applying any finite sequence of mutations is 2-acyclic. If $(Q,W)$ is a nondegenerate quiver with potential, then we can write $\mu_k(Q,W)=(Q',W')$ where $Q'=\mu(Q)$. That is, mutation of a quiver with potential agrees with mutation in the sense of Definition~\ref{def:quivermutation}.

The following theorem completely describes mutation of a quiver with potential in the class of examples studied in this paper.

\begin{theorem}[\cite{L}, Theorems~30,~31]
The quiver with potential associated to any ideal triangulation of $(\mathbb{D},\mathbb{M})$ is nondegenerate. Let $T$ and $T'$ be ideal triangulations related by a flip at some edge. Then, up to right equivalence, $(Q(T),W(T))$ and $(Q(T'),W(T'))$ are related by a mutation at the corresponding vertex.
\end{theorem}

\subsection{Hearts and tilting}

Let $\mathcal{D}$ be a $\Bbbk$-linear triangulated category. We will denote the shift functor on~$\mathcal{D}$ by $[1]$ and use the notation 
\[
\Hom_{\mathcal{D}}^i(A,B)\coloneqq\Hom_{\mathcal{D}}(A,B[i])
\]
for $A$,~$B\in\mathcal{D}$.

\begin{definition}
A \emph{t-structure} on $\mathcal{D}$ is a full subcategory $\mathcal{P}\subseteq\mathcal{D}$ satisfying the following:
\begin{enumerate}
\item $\mathcal{P}[1]\subseteq\mathcal{P}$.
\item Let $\mathcal{P}^\perp=\{G\in\mathcal{D}:\Hom(F,G)=0 \text{ for all $F\in\mathcal{P}$}\}$. Then for every object $E\in\mathcal{D}$, there is a triangle $F\rightarrow E\rightarrow G\rightarrow F[1]$ in~$\mathcal{D}$ with $F\in\mathcal{P}$ and $G\in\mathcal{P}^\perp$.
\end{enumerate}
A t-structure $\mathcal{P}\subseteq\mathcal{D}$ is said to be \emph{bounded} if 
\[
\mathcal{D}=\bigcup_{i,j\in\mathbb{Z}}\mathcal{P}^\perp[i]\cap\mathcal{P}[j].
\]
If $\mathcal{P}\subseteq\mathcal{D}$ is a t-structure, its \emph{heart} is defined as the intersection 
\[
\mathcal{A}=\mathcal{P}^\perp[1]\cap\mathcal{P}.
\]
In the following, when we talk about a \emph{heart} in~$\mathcal{D}$, we always mean the heart of some bounded t-structure.
\end{definition}

It is known that any heart in~$\mathcal{D}$ is a full abelian subcategory. A heart will be called \emph{finite length} if it is artinian and noetherian as an abelian category. Given full subcategories $\mathcal{A}$,~$\mathcal{B}\subseteq\mathcal{D}$, the \emph{extension closure} $\mathcal{C}=\langle\mathcal{A},\mathcal{B}\rangle\subseteq\mathcal{D}$ is defined as the smallest full subcategory of~$\mathcal{D}$ containing both $\mathcal{A}$ and~$\mathcal{B}$ such that if $X\rightarrow Y\rightarrow Z\rightarrow X[1]$ is a triangle in~$\mathcal{D}$ with $X$,~$Z\in\mathcal{C}$, then $Y\in\mathcal{C}$.

\begin{definition}
A pair of hearts $(\mathcal{A}_1,\mathcal{A}_2)$ in~$\mathcal{D}$ is called a \emph{tilting pair} if either of the equivalent conditions 
\[
\mathcal{A}_2\subseteq\langle\mathcal{A}_1,\mathcal{A}_1[-1]\rangle,
\quad
\mathcal{A}_1\subseteq\langle\mathcal{A}_2[1],\mathcal{A}_2\rangle
\]
is satisfied. In this case, we also say that $\mathcal{A}_1$ is a \emph{left tilt} of~$\mathcal{A}_2$ and that $\mathcal{A}_2$ is a \emph{right tilt} of~$\mathcal{A}_1$.
\end{definition}

If $(\mathcal{A}_1,\mathcal{A}_2)$ is a tilting pair in~$\mathcal{D}$, then the subcategories 
\[
\mathcal{T}=\mathcal{A}_1\cap\mathcal{A}_2[1], \quad \mathcal{F}=\mathcal{A}_1\cap\mathcal{A}_2
\]
form a torsion pair $(\mathcal{T},\mathcal{F})\subseteq\mathcal{A}_1$. Conversely, if $(\mathcal{T},\mathcal{F})\subseteq\mathcal{A}_1$ is a torsion pair, then the subcategory $\mathcal{A}_2=\langle\mathcal{F},\mathcal{T}[-1]\rangle$ is a heart, and the pair $(\mathcal{A}_1,\mathcal{A}_2)$ is a tilting pair.

We will be interested in a special case of the tilting construction. Suppose $\mathcal{A}$ is a finite length heart and $S\in\mathcal{A}$ a simple object. Let $\langle S\rangle\subseteq\mathcal{A}$ be the full subcategory consisting of objects $E\in\mathcal{A}$ all of whose simple factors are isomorphic to~$S$, and consider the full subcategories 
\[
S^\perp=\{E\in\mathcal{A}:\Hom_{\mathcal{A}}(S,E)=0\},
\quad
{^\perp S}=\{E\in\mathcal{A}:\Hom_{\mathcal{A}}(E,S)=0\}.
\]
Then the pairs $(\langle S\rangle,S^\perp)$ and $({^\perp S},\langle S\rangle)$ are torsion pairs.

\begin{definition}
The hearts 
\[
\mu_S^-(\mathcal{A})=\langle S[1],{^\perp S}\rangle,
\quad 
\mu_S^+(\mathcal{A})=\langle S^\perp,S[-1]\rangle
\]
are called the \emph{left tilt} and \emph{right tilt} of~$\mathcal{A}$ at~$S$, respectively.
\end{definition}

Now suppose $\mathcal{D}$ is a triangulated category satisfying the $\mathrm{CY}_3$ property 
\[
\Hom_{\mathcal{D}}^i(A,B)\cong\Hom_{\mathcal{D}}^{3-i}(B,A)^*
\]
for all objects $A$,~$B\in\mathcal{D}$. In addition, we will assume that $\mathcal{D}$ is algebraic in the sense of~\cite{Keller}. To a finite-length heart $\mathcal{A}\subseteq\mathcal{D}$, we can associate a quiver $Q(\mathcal{A})$ whose vertices are indexed by the isomorphism classes of simple objects $S_i\in\mathcal{A}$ with 
\[
\varepsilon_{ij}=\dim_{\Bbbk}\Ext_{\mathcal{A}}^1(S_i,S_j)
\]
arrows from the vertex~$i$ to the vertex~$j$. A finite-length heart is called \emph{nondegenerate} if it can be tilted indefinitely at all simple objects and in both directions without leaving the class of finite-length hearts, and if, for every heart $\mathcal{B}$ which is obtained from $\mathcal{A}$ by applying a sequence of tilts at simple objects, the quiver $Q(\mathcal{B})$ is 2-acyclic.

The quiver $Q(\mathcal{A})$ corresponding to a finite-length heart $\mathcal{A}$ has no loops if and only if every simple object is spherical in the sense of~\cite{SeidelThomas}. If $S\in\mathcal{D}$ is a spherical object, there is an associated autoequivalence $\Tw_S\in\Aut(\mathcal{D})$ called a \emph{spherical twist}~\cite{SeidelThomas}. If $\mathcal{A}$ is a nondegenerate finite-length heart containing $n$ simple objects $S_1,\dots,S_n$ up to isomorphism, then the spherical twists in these objects generate a group which we denote 
\[
\Sph_{\mathcal{A}}(\mathcal{D})=\langle\Tw_{S_1},\dots,\Tw_{S_n}\rangle\subseteq\Aut(\mathcal{D}).
\]
An important property of these spherical twist functors is the following.

\begin{proposition}
For every simple object $S\in\mathcal{A}$, we have 
\[
\Tw_S(\mu_S^-(\mathcal{A}))=\mu_S^+(\mathcal{A}).
\]
\end{proposition}

\subsection{The category associated to a quiver with potential}

Suppose $\mathcal{D}$ is a triangulated category with the $\mathrm{CY}_3$ property. We have seen how to associate to any finite-length heart $\mathcal{A}$ a quiver $Q(\mathcal{A})$. In fact, after choosing a suitable potential on~$Q(\mathcal{A})$, we can invert this construction.

\begin{theorem}[\cite{BridgelandSmith}, Theorem~7.2]
Let $(Q,W)$ be a quiver with reduced potential. Then there is an associated $\mathrm{CY}_3$ triangulated category $\mathcal{D}(Q,W)$, with a bounded t-structure whose heart 
\[
\mathcal{A}=\mathcal{A}(Q,W)\subseteq\mathcal{D}(Q,W)
\]
is of finite length, such that the associated quiver $Q(\mathcal{A})$ is isomorphic to~$Q$.
\end{theorem}

More explicitly, $\mathcal{D}(Q,W)$ is constructed as the subcategory of the derived category of the complete Ginzburg algebra of $(Q,W)$ consisting of objects with finite-dimensional cohomology. We refer to~\cite{KellerYang} for the details. For us, the most important property of this construction is its behavior under mutation.

\begin{theorem}[\cite{KellerYang}, Theorem~3.2, Corollary~5.5]
\label{thm:KellerYang}
Let $(Q,W)$ be a 2-acyclic quiver with potential, and $k$ a vertex of~$Q$. Suppose 
\[
(Q',W')=\mu_k(Q,W)
\]
is the quiver with potential obtained by mutation in the direction~$k$. Then there is a canonical pair of $\Bbbk$-linear triangulated equivalences 
\[
\Phi_\pm:\mathcal{D}(Q',W')\rightarrow\mathcal{D}(Q,W)
\]
which induce tilts in the simple object $S_k\in\mathcal{A}(Q,W)$ in the sense that 
\[
\Phi_\pm(\mathcal{A}(Q',W'))=\mu_{S_k}^\pm(\mathcal{A}(Q,W)).
\]
\end{theorem}

In particular, if $(Q(T),W(T))$ is the quiver with potential associated to an ideal triangulation~$T$, there is an associated category $\mathcal{D}(Q(T),W(T))$, and any two categories constructed from an ideal triangulation in this way are equivalent. Thus we get a $\mathrm{CY}_3$ triangulated category canonically associated to the disk $(\mathbb{D},\mathbb{M})$.

\subsection{Bridgeland stability conditions}

The notion of a stability condition on a triangulated category was introduced by Bridgeland~\cite{Bridgeland07} to formalize ideas about D-branes in topological string theory. Here we define a version of the space of stability conditions which is naturally identified with the space of framed quadratic differentials.

\begin{definition}[\cite{Bridgeland07}, Definition~1.1]
A \emph{stability condition} on a triangulated category $\mathcal{D}$ consists of a group homomorphism $Z:K(\mathcal{D})\rightarrow\mathbb{C}$ called the \emph{central charge} and a full additive subcategory $\mathcal{P}(\phi)\subseteq\mathcal{D}$ for each $\phi\in\mathbb{R}$ satisfying the following axioms:
\begin{enumerate}
\item If $E\in\mathcal{P}(\phi)$, then there exists $m(E)\in\mathbb{R}_{>0}$ such that $Z(E)=m(E)\exp(i\pi\phi)$.

\item If $\phi$ is any real number, then $\mathcal{P}(\phi+1)=\mathcal{P}(\phi)[1]$.

\item If $\phi_1>\phi_2$ and $A_i\in\mathcal{P}(\phi_i)$, then $\Hom_{\mathcal{D}}(A_1,A_2)=0$.

\item For each nonzero object $E\in\mathcal{D}$, there is a sequence of real numbers 
\[
\phi_1>\phi_2>\dots>\phi_k
\]
and a sequence of triangles 
\[
\xymatrix{
0=E_0 \ar[rr] & & E_1 \ar[ld] \ar[rr] & & E_2\ar[ld] \ar[r] & \dots \ar[r] & E_{k-1} \ar[rr] & & \ar[ld] E_k=E \\
& A_1 \ar@{-->}[lu] & & A_2 \ar@{-->}[lu] & & & & A_k \ar@{-->}[lu]
}
\]
such that $A_i\in\mathcal{P}(\phi_i)$ for all $i$.
\end{enumerate}
The homomorphism $Z$ is called the \emph{central charge}, and objects of $\mathcal{P}(\phi)$ are said to be \emph{semistable} of phase~$\phi$.
\end{definition}

In this paper, we will work with a category $\mathcal{D}$ for which the Grothendieck group $K(\mathcal{D})\cong\mathbb{Z}^{\oplus n}$ is free of finite rank, and we will restrict attention to stability conditions $\sigma=(Z,\mathcal{P})$ satisfying the \emph{support property} of~\cite{KS}: For some norm $\|\cdot\|$ on~$K(\mathcal{D})\otimes\mathbb{R}$, there is a constant $C>0$ such that 
\[
\|\gamma\|<C|Z(\gamma)|
\]
for all classes $\gamma\in K(\mathcal{D})$ represented by semistable objects in~$\mathcal{D}$. We write $\Stab(\mathcal{D})$ for the set of such stability conditions on~$\mathcal{D}$. One of the main properties of $\Stab(\mathcal{D})$ is that this set has a natural complex manifold structure.

\begin{theorem}[\cite{Bridgeland07}, Theorem~1.2]
The set $\Stab(\mathcal{D})$ has the structure of a complex manifold such that the forgetful map 
\[
\Stab(\mathcal{D})\rightarrow\Hom_{\mathbb{Z}}(K(\mathcal{D}),\mathbb{C})
\]
taking a stability condition to its central charge is a local isomorphism.
\end{theorem}

Any stability condition $\sigma=(Z,\mathcal{P})$ on~$\mathcal{D}$ has an associated heart 
\[
\mathcal{A}=\mathcal{P}((0,1])\subseteq\mathcal{D}
\]
defined as the extension closure of the subcategories $\mathcal{P}(\phi)$ for $0<\phi\leq1$. The central charge $Z$ maps all nonzero objects of~$\mathcal{A}$ to elements of the semi-closed upper half plane 
\[
\mathcal{H}=\{r\exp(i\pi\phi):r>0 \text{ and } 0<\phi\leq1\}\subseteq\mathbb{C}.
\]
Moreover, given a heart $\mathcal{A}\subseteq\mathcal{D}$ and a group homomorphism $Z:K(\mathcal{A})\rightarrow\mathbb{C}$ that maps into~$\mathcal{H}$, then assuming certain finiteness conditions, there is a unique stability condition on~$\mathcal{D}$ with central charge $Z$ such that the associated heart is~$\mathcal{A}$. In particular, if $\mathcal{A}$ is a finite length heart with simple objects $S_1,\dots,S_n$ up to isomorphism, then the subset $\Stab(\mathcal{A})\subseteq\Stab(\mathcal{D})$ consisting of stability conditions with associated heart~$\mathcal{A}$ is homeomorphic to 
\[
\mathcal{H}^n=\{Z\in\Hom_{\mathbb{Z}}(K(\mathcal{D}),\mathbb{C}):Z(S_i)\in\mathcal{H}\}.
\]
The following result explains how the chambers $\Stab(\mathcal{A})$ are glued together in $\Stab(\mathcal{D})$.

\begin{proposition}[\cite{BridgelandSmith}, Lemma~7.9]
Let $\mathcal{A}\subseteq\mathcal{D}$ be a finite length heart, and suppose $\sigma=(Z,\mathcal{P})\in\Stab\mathcal{D}$ lies on the boundary component of $\Stab(\mathcal{A})$ where $\Im(Z(S_i))=0$ for some chosen index~$i$. Assume that the tilted hearts $\mu_{S_i}^{\pm}(\mathcal{A})$ are also finite length. Then there is a neighborhood $\sigma\in U\subseteq\Stab(\mathcal{D})$ such that one of the following holds:
\begin{enumerate}
\item $Z(S_i)\in\mathbb{R}_{<0}$ and $U\subseteq\Stab(\mathcal{A})\cup\Stab(\mu_{S_i}^+(\mathcal{A}))$.
\item $Z(S_i)\in\mathbb{R}_{>0}$ and $U\subseteq\Stab(\mathcal{A})\cup\Stab(\mu_{S_i}^-(\mathcal{A}))$.
\end{enumerate}
\end{proposition}

Typically, we are interested in the quotient of $\Stab(\mathcal{D})$ by certain group actions. Note that the group $\Aut(\mathcal{D})$ of triangulated autoequivalences of~$\mathcal{D}$ acts on the space $\Stab(\mathcal{D})$: For any $\Phi\in\Aut(\mathcal{D})$ and $(Z,\mathcal{P})\in\Stab(\mathcal{D})$, we set $\Phi\cdot(Z,\mathcal{P})=(Z',\mathcal{P}')$ where 
\[
Z'(E)=Z(\Phi^{-1}(E))
\]
for all $E\in\mathcal{D}$ and $\mathcal{P}'=\Phi(\mathcal{P})$. It follows that any subgroup of $\Aut(\mathcal{D})$ acts on $\Stab(\mathcal{D})$. In addition to the action of autoequivalences of~$\mathcal{D}$, the space $\Stab(\mathcal{D})$ carries a natural $\mathbb{C}$-action. For any $z\in\mathbb{C}$ and $(Z,\mathcal{P})\in\Stab(\mathcal{D})$, we set $z\cdot(Z,\mathcal{P})=(Z',\mathcal{P}')$ where 
\[
Z'(E)=e^{-i\pi z}Z(E)
\]
for all $E\in\mathcal{D}$ and 
\[
\mathcal{P}'(\phi)=\mathcal{P}(\phi+\Re(z))
\]
for $\phi\in\mathbb{R}$. It is easy to see that this action commutes with the action of $\Aut(\mathcal{D})$.

Now let us specialize to the case where $(Q,W)$ is a quiver of type $A_n$ equipped with the zero potential and $\mathcal{D}=\mathcal{D}(Q,W)$ is the associated $\text{CY}_3$ triangulated category. Let $\mathcal{A}$ denote the distinguished heart of this category, and let $\Stab^\circ(\mathcal{D})$ denote the connected component of $\Stab(\mathcal{D})$ that contains $\Stab(\mathcal{A})$. There is an action of the group $\Sph_{\mathcal{A}}(\mathcal{D})$ on the space $\Stab(\mathcal{D})$. This action preserves the distinguished component, and so we can define 
\[
\Sigma(A_n)\coloneqq\Stab^\circ(\mathcal{D})/\Sph_\mathcal{A}(\mathcal{D}).
\]

Recall that there is a $\mathbb{C}^*$-action on the space of framed differentials where $t\in\mathbb{C}^*$ acts on a framed differential $(\phi,f_0)$ by modifying the coefficients of the polynomial defining $\phi$ and leaving the framing $f_0$ unchanged. Using this $\mathbb{C}^*$-action, we get an associated $\mathbb{C}$-action on this space, and we have the following result by~\cite{Ikeda} (see also~\cite{BridgelandSmith}).

\begin{theorem}
\label{thm:identifyquadstab}
There is a $\mathbb{C}$-equivariant biholomorphism 
\[
\Quad_{fr}(\mathbb{D},\mathbb{M})\stackrel{\sim}{\longrightarrow}\Sigma(A_n)
\]
where an element $z\in\mathbb{C}$ acts on the left hand side by taking a point $(\phi,f)\in\Quad_{fr}(\mathbb{D},\mathbb{M})$ to $e^{-2\pi i z/(n+3)}\cdot(\phi,f)$. Under this map, the set of all framed differentials that determine a given WKB triangulation $T$ is mapped isomorphically to a chamber in $\Sigma(A_n)$ consisting of stability conditions whose heart determines the quiver~$Q_T$.
\end{theorem}

\begin{proof}
It follows from Theorem~1.1 of~\cite{Ikeda} that there is an isomorphism $\mathbb{C}^n\setminus\Delta\cong\Sigma(A_n)$. Combining this with the isomorphism $\Quad_{fr}(\mathbb{D},\mathbb{M})\cong\mathbb{C}^n\setminus\Delta$ of Proposition~\ref{prop:identificationquad} gives the desired isomorphism. The results of~\cite{Ikeda} say that this map is equivariant with respect to the two $\mathbb{C}$-actions defined above.
\end{proof}

\subsection{The cluster Poisson variety}

Finally, let us define the cluster Poisson variety. Suppose $(Q_0,W_0)$ is a nondegenerate quiver with potential, and denote by $J$ the set of vertices of~$Q_0$. Let $\mathbb{T}_n$ denote an $n$-regular tree where $n=|J|$. We can label the edges of $\mathbb{T}_n$ by elements of $J$ in such a way that the $n$ edges emanating from any vertex have distinct labels. Choose a vertex $v_0$ of $\mathbb{T}_n$ and associate the quiver $Q_0$ to this vertex. We associate quivers to the remaining vertices in such a way that if two vertices are connected by an edge labeled $k$, then the quivers associated to these vertices are related by a mutation in the direction~$k$.

\begin{definition}
We say that two quivers $Q$ and $Q'$ are \emph{mutation equivalent} if $Q'$ is obtained from $Q$ by applying some sequence of mutations.
\end{definition}

Thus the quivers associated to vertices of $\mathbb{T}_n$ are all of the quivers mutation equivalent to~$Q_0$. For any quiver $Q$ which is mutation equivalent to~$Q_0$, let us consider the algebraic torus 
\[
\mathcal{X}_Q=(\mathbb{C}^*)^{J}.
\]
This torus carries a natural Poisson structure given by 
\[
\{X_i,X_j\}=\varepsilon_{ij}X_iX_j
\]
where $X_k$~($k\in J$) are the natural coordinates on~$\mathcal{X}_Q$ and 
\[
\varepsilon_{ij}=|\{\text{arrows from $i$ to $j$ in $Q$}\}| - |\{\text{arrows from $j$ to $i$ in $Q$}\}|.
\]
For each $k\in J$, there is a birational map $\mathcal{X}_Q\dashrightarrow\mathcal{X}_{\mu_k(Q)}$ preserving the Poisson structures. Abusing notation, we denote it by~$\mu_k$. If $X_j'$~($j\in J$) are the natural coordinates on $\mathcal{X}_{\mu_k(Q)}$, then $\mu_k$ is defined by the formula 
\[
\mu_k^*(X_j') = 
\begin{cases}
X_k^{-1} & \text{if $j=k$} \\
X_j{(1+X_k^{-\sgn(\varepsilon_{jk})})}^{-\varepsilon_{jk}} & \text{if $j\neq k$}.
\end{cases}
\]
If $v$ and $v'$ are any vertices of $\mathbb{T}_n$, there is a unique simple path from $v$ to $v'$. There are quivers $Q$ and $Q'$ associated to $v$ and $v'$, respectively. By composing the maps $\mu_k^*$ in order along the path connecting $v$ and $v'$, we obtain a birational map $\mathcal{X}_Q\dashrightarrow\mathcal{X}_{Q'}$. 

\begin{lemma}[\cite{GHK}, Proposition~2.4]
Let $\{Z_i\}$ be a collection of integral separated schemes of finite type over~$\mathbb{C}$ and suppose we have birational maps $f_{ij}:Z_i\dashrightarrow Z_j$ for all $i$, $j$ such that $f_{ii}$ is the identity and $f_{jk}\circ f_{ij}=f_{ik}$ as rational maps. Let $U_{ij}$ be the largest open subset of $Z_i$ such that $f_{ij}:U_{ij}\rightarrow f_{ij}(U_{ij})$ is an isomorphism. Then there is a scheme obtained by gluing the $Z_i$ along the open sets $U_{ij}$ using the maps $f_{ij}$.
\end{lemma}

Using this lemma, we can glue the tori $\mathcal{X}_Q$ to get a scheme.

\begin{definition}
Let $(Q_0,W_0)$ be a nondegenerate quiver with potential. Then the \emph{cluster Poisson variety} associated to $(Q_0,W_0)$ is the scheme obtained by gluing the tori $\mathcal{X}_Q$ for every quiver $Q$ mutation equivalent to~$Q_0$ using the above birational maps.
\end{definition}

Note that, despite the terminology, the cluster Poisson variety is not a variety in general but rather a nonseparated scheme.

As a special case of the above construction, we can take $Q_0$ to be a quiver of type $A_n$, and we get a cluster variety which we denote $\mathcal{X}(A_n)$. The quivers that are mutation equivalent to~$Q_0$ are precisely the quivers associated to ideal triangulations of the disk. In this case, the transformation used to glue to tori $\mathcal{X}_Q$ coincides with the transformation in Proposition~\ref{prop:transformcoordinates}. Thus we have an isomorphism $\mathcal{X}^*(\mathbb{D},\mathbb{M})\cong\mathcal{X}(A_n)$. Combining this with the identification $\Quad_{fr}(\mathbb{D},\mathbb{M})\cong\Sigma(A_n)$, we can reinterpret the map $H$ as a local biholomorphism 
\[
F:\Sigma(A_n)\rightarrow\mathcal{X}(A_n).
\]
For each $\hbar\in\mathbb{H}(\epsilon)$, one defines the associated map $F_\hbar$ using the $\mathbb{C}$-action on~$\Sigma(A_n)$, and Theorem~\ref{thm:introchambertotorus} follows from Theorems~\ref{thm:chambertotorus} and~\ref{thm:identifyquadstab}.

\section*{Acknowledgments}
\addcontentsline{toc}{section}{Acknowledgements}

The main result of this paper is a special case of joint work with Tom~Bridgeland, who contributed many important ideas. I am grateful to Kohei~Iwaki, Andrew~Neitzke, and Yoshitsugu~Takei for helpful discussions about exact WKB analysis.


\begin{thebibliography}{99}
\bibitem{KoikeSchafke} Koike, T. and Sch\"afke, R. On the Borel summability of WKB solutions of Schr\"odinger equations with polynomial potentials and its application. To appear.
\end{thebibliography}
\end{document}